\documentclass[12pt]{amsart}

\usepackage{fullpage,url,amssymb,enumerate,colonequals}

\usepackage{mathrsfs} 
\usepackage[section]{placeins}
\usepackage{MnSymbol}
\usepackage{extarrows}
\usepackage{lscape}
\usepackage[all,cmtip]{xy}

\usepackage[OT2,T1]{fontenc}
\usepackage{color}

\usepackage[
        colorlinks, citecolor=darkgreen,
        backref,
        pdfauthor={Nuno Freitas}, 
]{hyperref}

\usepackage{comment}
\usepackage{multirow}

\newcommand{\F}{\mathbb{F}}

\newcommand{\G}{\mathbb{G}}

\newcommand{\Q}{\mathbb{Q}}

\newcommand{\Z}{\mathbb{Z}}

\newcommand{\calB}{\mathcal{B}}

\newcommand{\vv}{\upsilon}

\DeclareMathOperator{\Aut}{Aut}

\DeclareMathOperator{\End}{End}

\DeclareMathOperator{\Frob}{Frob}
\DeclareMathOperator{\Gal}{Gal}

\newcommand{\unr}{{\operatorname{unr}}}

\newcommand{\GL}{\operatorname{GL}}

\newcommand{\PGL}{\operatorname{PGL}}

\newcommand{\SL}{\operatorname{SL}}

\numberwithin{equation}{section}

\newtheorem{theorem}{Theorem}
\newtheorem{lemma}{Lemma}
\newtheorem{corollary}{Corollary}

\theoremstyle{definition}

\theoremstyle{remark}

\definecolor{darkgreen}{rgb}{0,0.5,0}

\DeclareRobustCommand{\SkipTocEntry}[5]{}

\begin{document}

\title{On the degree of the $p$-torsion field 
of elliptic curves over $\Q_\ell$ for $\ell \neq p$}


\author{Nuno Freitas}
\address{Mathematics Institute\\
	University of Warwick\\
	CV4 7AL \\
	United Kingdom}
\email{nunobfreitas@gmail.com}

\author{Alain Kraus}
\address{Universit\'e Pierre et Marie Curie - Paris 6,
Institut de Math\'ematiques de Jussieu,
4 Place Jussieu, 75005 Paris, 
France}
\email{alain.kraus@imj-prg.fr}

\date{\today}
\thanks{The first-named author is
supported by the
European Union's Horizon 2020 research and innovation programme under the Marie Sk\l{l}odowska-Curie grant 
agreement No.\ 747808 and the grant {\it Proyecto RSME-FBBVA $2015$ Jos\'e Luis Rubio de Francia}.} 


\begin{abstract}
Let $\ell$ and $p \geq 3$ be distinct prime numbers. Let $E/\Q_{\ell}$ be an elliptic curve 
with $p$-torsion module $E_p$. Let $\Q_{\ell}(E_p)$ be the $p$-torsion field of $E$. 
We provide a complete description of the 
degree of the extension $\Q_{\ell}(E_p)/\Q_{\ell}$. As a consequence,  we obtain a  recipe to determine 
the discriminant ideal of the extension $\Q_{\ell}(E_p)/\Q_\ell$ in terms of standard information on $E$.
\end{abstract}

\maketitle

\section{Introduction}
Let $\ell$ and $p \geq 3$ be distinct prime numbers. Fix $\overline{\Q_{\ell}}$ an algebraic closure of $\Q_{\ell}$. Let~$E/\Q_{\ell}$ be an elliptic curve 
with $p$-torsion module $E_p$. Let $\Q_{\ell}(E_p) \subset \overline{\Q}_{\ell}$ be the $p$-torsion field of~$E$. The aim of this paper is to 
determine the degree $d$ of the extension $\Q_{\ell}(E_p)/\Q_{\ell}$. 

Write $\pi$ for an uniformizer in $\Q_{\ell}(E_p)$ and $e$ for its ramification degree. The different 
ideal of~$\Q_{\ell}(E_p)$ is~$(\pi)^D$, where the integer $D$ is fully determined in \cite{CaliKraus}. 
The discriminant ideal~$\mathcal{D}$ of the extension $\Q_{\ell}(E_p)/\Q_\ell$ is generated by~$\ell^{dD/e}$. 
The value of $e$   is given in \cite{Kraus} in terms of the standard invariants of a minimal Weierstrass
model of $E$. Therefore, as a consequence of our results, we obtain a complete procedure to determine $\mathcal{D}$ in terms of $\ell$, $p$ and invariants attached to~$E$.

{\large \part{Statement of the  results}}

Let $\ell$ and $p \geq 3$ be distinct prime numbers. Let $v$ be the valuation in $\Q_{\ell}$ such that $v(\ell)=1$. 

Let $E/\Q_{\ell}$ be an elliptic curve and write
$c_4$, $c_6$ and $\Delta$ for the standard invariants of a minimal Weierstrass model of $E/\Q_{\ell}$. 
Write $j = \frac{c_4^3}{\Delta}$ for the modular invariant of~$E$. 

Write $E_p$ for the $p$-torsion module of $E$. 
Let $\Q_{\ell}(E_p) \subset \overline{\Q}_{\ell}$ be the $p$-torsion field of $E$ and denote 
by $d$ its degree
$$d=[\Q_{\ell}(E_p):\Q_{\ell}].$$

Write $\Q_\ell^{\unr}$ for the 
maximal unramifed extension of $\Q_\ell$ contained in $\overline{\Q}_{\ell}$.

Let $r$ be the order of $\ell$ modulo $p$ and $\delta$ be the order of $-\ell$ modulo $p$. 

We will now state our results according to the type of reduction of $E$ and the prime~$\ell$.

\section{The case of good reduction}

Let $M/\Q_{\ell}$ be a finite  totally ramified extension and $E/M$ be an elliptic curve with good reduction.
In this section we will write $d$ to denote 
$$d=[M(E_p):M].$$
The residue field of $M$ is $\F_{\ell}$. Let $\tilde{E}/\F_{\ell}$ be the elliptic curve obtained from $E/M$ by reduction of a minimal model. Write
$|\tilde{E}(\F_{\ell})|$ for the order of its  group of  $\F_{\ell}$-rational points. Let 
$$a_E=\ell+1-|\tilde{E}(\F_{\ell})| \quad \text{and}\quad  
\Delta_E=a_E^2-4\ell.$$
Note that the Weil bound implies $\Delta_E<0$.

In the case of $E/M$ having good ordinary reduction i.e. $a_E\not\equiv 0 \pmod \ell$, the endomorphism ring  of $\tilde E$ is isomorphic to an order of the imaginary quadratic field 
$\Q\left(\sqrt{\Delta_E}\right)$. The Frobenius endomorphism $\pi_{\tilde{E}}$ of $\tilde E$, generates a subring $\Z[\pi_{\tilde{E}}]$ of finite index of $\End(\tilde E)$. 
In this case, we define
\begin{equation}
\label{(2.1)} 
b_E=\left[\End(\tilde{E}):\Z[\pi_{\tilde{E}}]\right].
\end{equation}
The ratio $\frac{ \Delta_E}{b_E^2}$ is the discriminant of $\End(\tilde{E})$. 
The determination of $b_E$ has been implemented on  {\tt Magma} \cite{MAGMA}
by Centeleghe  (cf.  \cite{Centeleghe} and \cite{CentelegheProgram}).  A method  
to obtain $b_E$ is also presented in \cite{Duketoth}.

Let $\alpha$ and $\beta$ be the roots in $\F_{p^2}$ of the polynomial in $\F_p[X]$ given by
$$X^2-a_EX+\ell \pmod p.$$
  
\begin{theorem} \label{T:thm1}
Let  $n$ be the least common multiple of the orders of $\alpha$ and $\beta$ in  $\F_{p^2}^*$. 
\begin{itemize}
 \item[1)] If  $\Delta_E\not\equiv 0   \pmod p$, then $d=n$.
  \item[2)] Suppose    $\Delta_E\equiv 0 \pmod p$. Then $E/M$ has good ordinary reduction. We have
  \[
d= 
  \begin{cases}
n \ \text{ if } \ b_{E}\equiv 0 \pmod p, \\
np   \ \text{ otherwise}.\
  \end{cases}
\]
\end{itemize}
\end{theorem}

\begin{corollary} \label{C:cor1}
Suppose $a_E=0$. Then $d=2\delta$.
\end{corollary}

\section{The case of multiplicative reduction}

Let $E/\Q_{\ell}$ be an elliptic curve with multiplicative reduction. In this case, we have $v(c_6)=0$ and the Legendre symbol 
$\left(\frac{-c_6}{\ell}\right)=\pm 1.$
Let
$$\tilde{j}=\frac{j}{\ell^{v(j)}}.$$

\begin{theorem} \label{T:thm2}
We are in one of  the following cases.

\begin{itemize}

 \item[1)] Suppose $\ell\geq 3$ and $\left(\frac{-c_6}{\ell}\right)=1$, or $\ell=2$ and $c_6\equiv 7\pmod 8$.
\medskip

\item[1.1)] If $\ell\not\equiv 1 \pmod p$, then
    \[
d= 
  \begin{cases}
r \ \text{ if } \ v(j)\equiv 0 \pmod p, \\
pr   \ \text{ otherwise}.\
  \end{cases}
\]
 \smallskip

 \item[1.2)] If $\ell\equiv 1 \pmod p$, then
    \[
d= 
  \begin{cases}
1 \ \text{ if } \ v(j)\equiv 0 \pmod p\quad \text{and}\quad  \ \tilde j^{\frac{\ell-1}{p}}\equiv 1 \pmod \ell, \\
p   \ \text{ otherwise}.\
  \end{cases}
\]
\medskip
 \item[2)] Suppose $\ell\geq 3$ and $\left(\frac{-c_6}{\ell}\right)=-1$, or $\ell=2$ and $c_6\not\equiv 7 \pmod 8$. 
\medskip

\item[2.1)] If $r$ is even, then
    \[
d= 
  \begin{cases}
r \ \text{ if } \ v(j)\equiv 0 \pmod p, \\
pr   \ \text{ otherwise}.\
  \end{cases}
\]
 \smallskip

\item[2.2)] Suppose $r$ odd. 
\medskip
 \item[2.2.1)] If $\ell\not\equiv 1 \pmod p$, then
\[
d= 
  \begin{cases}
2r \ \text{ if } \ v(j)\equiv 0 \pmod p, \\
2pr   \ \text{ otherwise}.\
  \end{cases}
\]
 \smallskip
 \item[2.2.2)] If $\ell\equiv 1 \pmod p$, then
    \[
d= 
  \begin{cases}
2 \ \text{ if } \ v(j)\equiv 0 \pmod p\quad \text{and}\quad  \ \tilde j^{\frac{\ell-1}{p}}\equiv 1 \pmod \ell, \\
2p   \ \text{ otherwise}.\
  \end{cases}
\]
\end{itemize}
\end{theorem}

\section{The case of additive potentially multiplicative reduction}
Let us assume that $E/\Q_{\ell}$ has additive potentially multiplicative reduction. 

\begin{theorem} \label{T:thm3}
We are in one of  the following cases.
\begin{itemize}
 \item[1)] If $\ell\not\equiv 1 \pmod p$, then
\[
d= 
  \begin{cases}
2r \ \text{ if } \ v(j)\equiv 0 \pmod p, \\
2pr   \ \text{ otherwise}.\
  \end{cases}
\]
 \smallskip
 \item[2)] If $\ell\equiv 1 \pmod p$, then
    \[
d= 
  \begin{cases}
2 \ \text{ if } \ v(j)\equiv 0 \pmod p\quad \text{and}\quad  \ \tilde j^{\frac{\ell-1}{p}}\equiv 1 \pmod \ell, \\
2p   \ \text{ otherwise}.\
  \end{cases}
\]
\end{itemize}
\end{theorem}

\section{The case of additive potentially good reduction with $\ell\geq 5$}
Let $\ell\geq 5$ and $E/\Q_{\ell}$ be an elliptic curve with additive potentially good reduction. 

In this case, the triples $\left(v(c_4),v(c_6),v(\Delta)\right)$ 
are given according to the following table.

\bigskip
\centerline{\vbox{\offinterlineskip
\halign{\vrule height12pt\ \hfil#\hfil\ \vrule&&\ \hfil#\hfil\ \vrule\cr
\noalign{\hrule}
\omit\vrule height2pt\hfil\vrule&&&&&&&&\cr
$v(\Delta)$ &2 &3&4&6&6&8&9&10\cr
\omit\vrule height2pt\hfil\vrule&&&&&&&&\cr
\noalign{\hrule}
$v(c_4)$ &$\geq 1$&1&$\geq 2$&2&$\geq 3$&$\geq 3$&3&$\geq 4$\cr
\omit\vrule height2pt\hfil\vrule&&&&&&&&\cr
\noalign{\hrule}
$v(c_6)$ &1&$\geq 2$&$2$&$\geq 3$&3&4&$\geq 5$&5\cr
\omit\vrule height2pt\hfil\vrule&&&&&&&&\cr
\noalign{\hrule}
}
}}
\bigskip
\vskip0pt\noindent

Let $e=e(E)$ be the semistability defect of $E$, i.e. the degree of the 
minimal extension of~$\Q^{\unr}_\ell$ over which  $E$ acquires good reduction.

From \cite[Proposition~1]{Kraus} we know that
\begin{equation}
\label{(5.1)}
e=\text{denominator of}\  \frac{v(\Delta)}{12}
\end{equation}
and, in particular, we have $e\in \lbrace 2,3,4,6\rbrace$.
The equation
\begin{equation}
\label{(5.2)} 
y^2=x^3-\frac{c_4}{48}x-\frac{c_6}{864}.
\end{equation}
is a minimal model of $E/\Q_{\ell}$. 

\subsection{Case $e=2$} Suppose $E$ satisfies $e=2$.
Let $E'/\Q_{\ell}$ be the quadratic twist of $E$ by~$\sqrt{\ell}$. 

\begin{lemma}  \label{L:lemma1}  
The elliptic curve $E'/\Q_{\ell}$ has good reduction.
\end{lemma}

Let $\tilde{E'}/\F_{\ell}$ be the elliptic curve obtained by reduction of a minimal model for~$E'$. Write
\begin{equation}
a_{E'}=\ell+1-|\tilde {E'}(\F_{\ell})|\quad \text{and}\quad \Delta_{E'}=a_{E'}^2-4\ell.
\end{equation}

Let $\alpha$ and $\beta$ be the roots in $\F_{p^2}^*$ of the polynomial in $\F_p[X]$
$$X^2-a_{E'}X+\ell \pmod p.$$
 Let  $n$ be the least common multiple of the orders of $\alpha$ and $\beta$ in  $\F_{p^2}^*$. 

In case $E'/\Q_{\ell}$ has good ordinary reduction, $\pi_{\tilde{E}'}$ being the Frobenius endomorphism of $\tilde{E'}/\F_{\ell}$, we will note 
$$b_{E'}=\left[\End(\tilde{E'}):\Z[\pi_{\tilde{E}'}]\right]$$
the index of $\Z[\pi_{\tilde{E}'}]$ in $\End(\tilde{E'})$.

\begin{theorem}  \label{T:thm4} 
Let $E/\Q_\ell$ satisfy $e(E)=2$ and let $E'$ be as above.

\begin{itemize}

 \item[1)] Suppose $\Delta_{E'}\not\equiv 0 \pmod p$. We have
  \[
d= 
  \begin{cases}
n \ \text{ if }  \ n \ \text{is even}\quad \text{and } \quad  \alpha^{{n\over 2}}=\ \beta^{{n\over 2}}=-1, \\
2n  \ \text{ otherwise}.\
  \end{cases}
\]

 \smallskip
 \item[2)] Suppose $\Delta_{E'}\equiv 0 \pmod p$. Then $E'/\Q_{\ell}$ has good ordinary reduction.
 \medskip
 
  \item[2.1)] If $n$ is even and $\alpha^{\frac{n}{2}}=-1$, then
   \[
d= 
  \begin{cases}
n \ \text{ if } \  b_{E'}\equiv 0 \pmod p, \\
np  \ \text{ otherwise}.\
  \end{cases}
\]
 \item[2.2)] If $n$ is odd or $\alpha^{\frac{n}{2}}\neq -1$, then
   \[
d= 
  \begin{cases}
2n \ \text{ if }  \ b_{E'}\equiv 0 \pmod p, \\
2np  \ \text{ otherwise}.\
  \end{cases}
\]

\end{itemize}
\end{theorem}

\subsection{Case $e \in \lbrace 3,4,6\rbrace$}  
Suppose $E/\Q_\ell$ satisfies $e \in \lbrace 3,4,6\rbrace$.
We define
\begin{equation}
\label{(5.4)}
u=\ell^{{v(\Delta)\over 12}}\quad \hbox{and}\quad  M=\Q_{\ell}(u).
\end{equation}
Let  $E'/M$ be the elliptic curve  of  equation 
\begin{equation}
\label{(5.5)} 
Y^2=X^3-{c_4\over 48u^4}X-{c_6\over 864u^6}.
\end{equation}

\begin{lemma} \label{L:lemma2} The elliptic curves $E$ and $E'$ are isomorphic over $M$. Moreover,  $E'/M$ has good reduction.
\end{lemma}
The extension $M/\Q_{\ell}$ is totally ramified of degree $e$.  Let $\tilde{E'}/\F_{\ell}$ be the elliptic curve obtained from $E'$ by reduction
and denote
\begin{equation}
a_{E'}=\ell+1-|\tilde {E'}(\F_{\ell})|\quad \text{and}\quad \Delta_{E'}=a_{E'}^2-4\ell.
\end{equation}
Let $\alpha$ and $\beta$ be the roots in $\F_{p^2}^*$ of the polynomial in $\F_p[X]$ given by
$$X^2-a_{E'}X+\ell \pmod p.$$
Let  $n$ be the least common multiple of the orders of $\alpha$ and $\beta$ in  $\F_{p^2}^*$. 
 
When $p$ does not divide $e$,  we denote by $\zeta_e$ a primitive $e$-th root of unity in $\F_{p^2}^*$. Note that when $p \mid e$ we have $p=3$ and $e\in \lbrace 3,6\rbrace$.

\begin{theorem} \label{T:thm5}  
Suppose that $e(E)=3$ and $\ell\equiv 1 \pmod 3$.
\begin{itemize}

\item[1)]  Assume also $p\neq 3$. 
\medskip

\item[1.1)] If  $\Delta_{E'}\not\equiv 0 \pmod p$, then
\[
d= 
  \begin{cases}
n \ \text{ if } \ n\equiv 0 \pmod 3 \quad \text{and} \quad  \lbrace \alpha^{{n\over 3}},\ \beta^{{n\over 3}}\rbrace=\lbrace \zeta_3,\zeta_3^{-1}\rbrace, \\
3n  \ \text{ otherwise}.\
  \end{cases}
\]
\smallskip
\item[1.2)] If $\Delta_{E'}\equiv 0 \pmod p$, then $d=3n$.
\medskip

\item[2)] If $p=3$, then $d=3n$.

\end{itemize}
\end{theorem}

\begin{theorem}   \label{T:thm6}  
Suppose that $e(E)=4$ and $\ell\equiv 1 \pmod 4$.
\begin{itemize}

\item[1)]  If $\Delta_{E'}\not\equiv 0 \pmod p$, then
\medskip
 \[
d= 
  \begin{cases}
n \ \text{ if } \ n\equiv 0 \pmod 4   \quad \text{and} \quad  \lbrace \alpha^{{n\over 4}}, \beta^{{n\over 4}}\rbrace=\lbrace \zeta_4,\zeta_4^{-1}\rbrace, \\
4n  \ \text{ if } \  n  \ \text{is odd} \quad \text{or} \quad  \lbrace \alpha^{{n\over 2}},\beta^{{n\over 2}}\rbrace\neq \lbrace -1\rbrace,\\
2n \ \text{ otherwise}.\
  \end{cases}
\]
\smallskip

\item[2)]  If $\Delta_{E'}\equiv 0 \pmod p$, then
\smallskip

  \[
d= 
  \begin{cases}
2n  \ \text{ if }  n  \ \text{is even}\quad  \text{and}\quad \alpha^{{n\over2}}=-1, \\
4n \ \text{otherwise}.\
  \end{cases}
\]

\end{itemize}
\end{theorem}

\bigskip

\begin{theorem}  \label{T:thm7} 
Suppose that $e(E)=6$ and $\ell\equiv 1 \pmod 3$.
\begin{itemize}
\item[1)] Assume also $p\neq 3$.   
\bigskip
\item[1.1)]  Suppose  $\Delta_{E'}\not\equiv 0 \pmod p$. 
\bigskip

\item[1.1.1)]  If $n\equiv 0 \pmod 6$ and   $\lbrace \alpha^{{n\over 6}}, \beta^{{n\over 6}}\rbrace=\lbrace \zeta_6, \zeta_6^{-1}\rbrace$, then  $d=n$.
\medskip
 \item[1.1.2)]  Suppose $n\not\equiv 0 \pmod 6$ or $\lbrace \alpha^{{n\over 6}}, \beta^{{n\over 6}}\rbrace\neq \lbrace \zeta_6, \zeta_6^{-1}\rbrace$. Then,
 \[
d= 
  \begin{cases}
2n \ \text{ if } \ n\equiv 0 \pmod 3 \quad \text{and} \quad  \lbrace \alpha^{{n\over 3}}, \beta^{{n\over 3}}\rbrace=\lbrace \zeta_6^2,\zeta_6^{-2}\rbrace, \\
3n  \ \text{ if } \  n  \ \text{is even} \quad \text{and} \quad   \alpha^{{n\over 2}}=\beta^{{n\over 2}}=-1,\\
6n \ \text{ otherwise.}\
  \end{cases}
\]
\smallskip

\item[1.2)]  If  $\Delta_{E'}\equiv 0 \pmod p$, then
  \[
d= 
  \begin{cases}
3n  \ \text{ if }  n  \ \text{is even}\quad \text{and}\quad \alpha^{{n\over2}}=-1, \\
6n \ \text{otherwise}.\
  \end{cases}
\]

\item[2)] If  $p=3$, then $d=6$.
\end{itemize}
\end{theorem}

\begin{theorem}  \label{T:thm8}  
Suppose that $e=e(E)\in \lbrace 3,4,6\rbrace$ and $\ell\equiv -1 \pmod e$.  
\begin{itemize}
\item[1)] If $e=3$, then   $d=6\delta$.
\smallskip

\item[2)] If $e\in \lbrace 4, 6\rbrace$, then
  \[
d= 
  \begin{cases}
er  \ \text{ if } r  \ \text{is even}, \\
2er  \ \text{ if } r   \ \text{is odd}.\
  \end{cases}
\]
\end{itemize}
\end{theorem}

\section{The case of additive potentially good reduction with $\ell=3$}
Let $E/\Q_3$ be an elliptic curve with 
additive potentially good reduction. 
We can find in \cite{Kraus} the value of $e$ in terms of the triple $\left(v(c_4),v(c_6),v(\Delta)\right)$. In particular, 
we have $e\in \lbrace 2,3,4,6,12\rbrace$.

When $e=2$, we see from \cite[p. 355, Cor.]{Kraus} that
$$\left(v(c_4),v(c_6),v(\Delta)\right)\in \lbrace  (2,3,6), (3,\geq 6,6) \rbrace.$$
In this case, a minimal equation of $E/\Q_3$ is
\begin{equation}
\label{(6.1)} 
y^2=x^3-\frac{c_4}{48}x-\frac{c_6}{864}
\end{equation}
and we 
let $E'/\Q_3$ be the elliptic curve obtained as the quadratic twist of $E$ 
by~$\sqrt{3}$. 

\begin{lemma}  \label{L:lemma3} The elliptic curve $E'/\Q_3$ has good reduction.
\end{lemma}

Let $\tilde{E'}/\F_3$ be the elliptic curve obtained from $E'/\Q_3$ by reduction and define
$$a_{E'}=4-|\tilde {E'}(\F_3)|\quad \text{and} \quad \Delta_{E'}=a_{E'}^2-12.$$
Let $\alpha$ and $\beta$ be the roots in $\F_{p^2}^*$ of the polynomial in $\F_p[X]$ given by
$$X^2-a_{E'}X+3 \pmod p.$$
Let $n$ be the least common multiple of their orders  in  $\F_{p^2}^*$

\begin{theorem}  \label{T:thm9} Let $E/\Q_3$ satisfy $e(E)=2$ and let $E'$ be as above.
\begin{itemize}
\item[1)]  Suppose  $\Delta_{E'}\not\equiv 0 \pmod p$. 
We have
  \[
d= 
  \begin{cases}
n \ \text{ if }  n \ \text{is even}\quad \text{and } \quad  \alpha^{{n\over 2}}=\ \beta^{{n\over 2}}=-1, \\
2n  \ \text{ otherwise}.\
  \end{cases}
\]

\item[2)]  Suppose $\Delta_{E'}\equiv 0 \pmod p$. Then $p=11$ and $d=110$.
\end{itemize}
\end{theorem}

\medskip

\begin{theorem}  \label{T:thm10} 
Suppose that $e(E) \in \lbrace 3,4,6, 12\rbrace$.
\begin{itemize}
\item[1)] If $e=3$, then $d=6\delta$.
\smallskip

\item[2)] If $e\in \lbrace 4, 6, 12 \rbrace$, then
  \[
d= 
  \begin{cases}
er  \ \text{ if } r  \ \text{is even}, \\
2er  \ \text{ if } r   \ \text{is odd}.\
  \end{cases}
\]
\end{itemize}
\end{theorem}

\section{The case of additive   potentially good reduction with $\ell=2$}
Let $E/\Q_2$ be an elliptic curve with additive potentially good reduction. 
We can find in \cite{Kraus} the value of $e$ in terms of the triple $\left(v(c_4),v(c_6),v(\Delta)\right)$. In particular, 
$e\in \lbrace 2,3,4,6,8,24\rbrace$.

When $e=2$, we write $t=\left(v(c_4),v(c_6),v(\Delta)\right)$, 
and \cite[p. 357, Cor.]{Kraus} gives that
\begin{equation}
t\in \lbrace (\geq 6,6,6), (4,6,12), (\geq 8,9,12), (6,9,18)\rbrace.
\end{equation}
The equation
\begin{equation}
\label{(7.2)} 
y^2=x^3-{c_4\over 48}x-{c_6\over 864}
\end{equation}
is a minimal model of $E/\Q_2$. Define the quantity
$$c_6'=\frac{c_6}{2^{v(c_6)}}.$$

\begin{lemma} \label{L:lemma4} Suppose that $E$ satisfies $e=2$ and 
let $u\in \lbrace -2,-1,2\rbrace$ be defined as follows:
 \[
u = 
  \begin{cases}
  2 \ \text{ if } \ t=(\geq 6,6,6) \quad \text{and} \quad c'_6\equiv 1 \pmod 4, \\
  -2 \ \text{ if } \ t=(\geq 6,6,6) \quad \text{and} \quad c'_6\equiv -1 \pmod 4, \\
  -1 \ \text{ if } \ t=(4,6,12) \quad \text{or} \quad t=(\geq 8,9,12), \\
  2 \ \text{ if } \ t=(6,9,18) \quad \text{and} \quad c'_6\equiv -1 \pmod 4, \\
 -2 \ \text{ if } \ t=(6,9,18) \quad \text{and} \quad c'_6\equiv 1 \pmod 4. \
  \end{cases}
\]
Then, the quadratic twist of $E/\Q_{2}$ by $\sqrt{u}$ has good reduction.
\end{lemma}
Under the conditions of Lemma~\ref{L:lemma4}, we 
let $E'/\Q_2$ be the quadratic twist of $E/\Q_{2}$ by $\sqrt{u}$ and 
 $\tilde{E'}/\F_2$ the elliptic curve obtained from $E'$  by reduction. 
Define also
$$a_{E'}=3-|\tilde {E'}(\F_2)|\quad \text{and} \quad \Delta_{E'}=a_{E'}^2-8.$$
Let $\alpha$ and $\beta$ be the roots in $\F_{p^2}^*$ of the polynomial in $\F_p[X]$ given by
$$X^2-a_{E'}X+2 \pmod p.$$
Let  $n$ be the least common multiple of their orders  in  $\F_{p^2}^*$.

\begin{theorem}  \label{T:thm11}  
Let $E/\Q_2$ satisfy $e(E)=2$ and let $E'$ be as above.
\begin{itemize}

\item[1)]   Suppose  $\Delta_{E'}\not\equiv 0 \pmod p$. 
\medskip
We have,
  \[
d= 
  \begin{cases}
n \ \text{ if }  n \ \text{is even}\quad \text{and} \quad  \alpha^{{n\over 2}}=\ \beta^{{n\over 2}}=-1, \\
2n  \ \text{ otherwise}.\
  \end{cases}
\]
\item[2)]  Suppose $\Delta_{E'}\equiv 0 \pmod p$. Then $p=7$ and $d=42$.
\end{itemize}
\end{theorem}

\begin{theorem} \label{T:thm12} 
Suppose that $E/\Q_2$ satisfies $e(E) \in \lbrace 3,4,6,8,24\rbrace$.
\begin{itemize}
\item[1)] If $e=3$, then $d=6\delta$.
\smallskip
\item[2)] If $e\in \lbrace 4, 6, 8, 24 \rbrace$, then
  \[
d= 
  \begin{cases}
er  \ \text{ if } r  \ \text{is even}, \\
2er  \ \text{ if } r   \ \text{is odd}.\
  \end{cases}
\]
\end{itemize}
\end{theorem}

\section{Application of the results}

Before proceeding to the proofs of the results, let us give some examples of 
their application.

Consider the elliptic curve over $\Q$ with Cremona label $25920ba1$ given by the minimal model 
\[
 E \; : \; y^2 = x^3 - 432x - 864 
\]
whose conductor is $N_E = 2^6 \cdot 3^4 \cdot 5$ and standard invariants are 
\[
 c_4(E) = 2^8 \cdot 3^4, \qquad c_6(E) = 2^{10} \cdot 3^6, \qquad \Delta = 2^{14} \cdot 3^{10} \cdot 5.
\]
We will determine the degree $d_p$ 
of $\Q_\ell(E_p)/\Q_\ell$ for $\ell \in \{2,3,5,7 \}$ 
and $p \in \{3, 5, 7, 11 \}$ with $\ell \neq p$.
Recall that $r$ is the order of $\ell$ modulo~$p$.

1) Let $\ell=2$; from \cite[p. 357, Cor.]{Kraus} we see that $e=24$. 
From Theorem~\ref{T:thm12} we now conclude that 
\[
 d_3 = 48, \qquad d_5 = 96, \qquad d_7 = 144, \qquad d_{11} = 240,
\]
since for $p=3,5,7,11$ we have $r = 2, 4, 3, 10$, respectively.

2) Let $\ell=3$; from \cite[p. 355, Cor.]{Kraus} we see that $e=6$. 
From Theorem~\ref{T:thm10} we now conclude that 
\[
 d_5 = 24, \qquad d_7 = 36, \qquad d_{11} = 110,
\]
since for $p=5,7,11$ we have $r = 4, 6, 5$, respectively.

3) Let $\ell=5$; the curve $E$ has multiplicative reduction at $5$ and 
the symbol $(-c_6/5) = 1$. Moreover, for $p=3,7,11$, 
we have that $5 \not\equiv 1 \pmod{p}$, so by part 1.1) 
of Theorem~\ref{T:thm2} we conclude that 
\[
 d_3 = 6, \qquad d_7 = 42, \qquad d_{11} = 55, 
\]
since $r = 2, 6, 5$, respectively
and $\vv_5(j) = 2 \not\equiv 0 \pmod{p}$ for all $p$.

4) Let $\ell=7$; the curve $E$ has good reduction at~$7$ so we will apply Theorem~\ref{T:thm1}.
We have
\[
 a_E = -2, \qquad \Delta_E = -24.
\]
For $p=3$ we have $\Delta_E \equiv 0 \pmod{p}$ and $b_E = 1 \not\equiv 0 \pmod{3}$. 
Moreover, 
\[ x^2 - a_E x + 7 \equiv (x+1)^2 \pmod{3}, \] 
so $n=2$ and $d_3 = 6$ by Theorem~\ref{T:thm1} part 2). 
For $p =5$, we have 
\[ x^2 - a_E x + 7 \equiv (x+3)(x+4) \pmod{5}, \] 
so $d_5=n=4$ by Theorem~\ref{T:thm1} part 1). 
For $p=11$, we have 
\[ x^2 - a_E x + 7 \equiv (x+8)(x+4) \pmod{11}, \] 
so $d_{11}=n=10$ by Theorem~\ref{T:thm1} part 1). 

\subsection{Computing the discriminant of $\Q_\ell(E_p)$}
To finish we will determine the discriminant ideal of $K = \Q_\ell(E_p)$
for $(\ell,p)=(2,3)$ and $(\ell,p)=(3,5)$ where $E$ is the elliptic curve 
in the previous examples. Write $\pi$ for an uniformizer in $K$. 
The discriminant ideal~$\mathcal{D}$ of $K/\Q_\ell$ is generated by~$\ell^{dD/e}$, 
where $(\pi)^D$ is the different ideal of $K$.

1) Let $(\ell,p)=(2,3)$. We have $e=24$ and $d_3 = 48$ from the previous section. 
The valuations of $\vv_2(c_4)$ and $\vv_2(\Delta)$ together with
\cite[Th\'eor\`eme~4]{CaliKraus} tell us that $D=50$, hence $\mathcal{D} = (2)^{100}$.

2) Let $(\ell,p)=(3,5)$. In this case, we have $e=6$ and $d_5 = 24$ and
\cite[Th\'eor\`eme~3]{CaliKraus} tell us that $D=9$, hence $\mathcal{D} = (3)^{36}$.

{\large \part{Proof of the statements}}

\section{Proof of Theorem~\ref{T:thm1} and Corollary~\ref{C:cor1}}

Let $E/M$ be as in the statement of Theorem~\ref{T:thm1} 
and recall that $d=[M(E_p):M]$.

\begin{lemma} \label{L:lemma5}   
Suppose $p$ divides $\Delta_E$. Then $E/M$ has good ordinary reduction. 
\end{lemma}

\begin{proof} We have to  prove that $a_E\not\equiv 0\pmod \ell$. 

Recall that $\Delta_E = a_E^2 -4\ell$. Since $\ell\neq p$ and $p\geq 3$, we have $a_{E}\neq 0$. 

The Weil bound implies $|a_E|\leq 2\sqrt{\ell}$. 
So, for $\ell\geq 5$, it is clear that $\ell \nmid a_E$.

Suppose $\ell=2$ or $\ell=3$ and that $E/M$ has good supersingular reduction. 
In case $\ell=2$, one has  $a_E=\pm 2$ so   $\Delta_E=-4$,  which is not divisible by $p$. 
If $\ell=3$, 
one has
$a_{E}=\pm 3$ so $\Delta_E=-3$,  which leads again to a contradiction, hence the assertion. 
\end{proof}

Let $\tilde E_p$ be the group of $p$-torsion points of the reduced elliptic curve $\tilde E/\F_{\ell}$. From the work of Serre-Tate \cite[Lemma 2]{ST1968}, we have
\begin{equation}
d=[\F_{\ell}(\tilde E_p):\F_{\ell}].
\end{equation}
Let us note  
$$\rho_{\tilde E,p} : \Gal(\F_{\ell}(\tilde E_p)/\F_{\ell}) \to \GL_2(\F_p)$$ 
the representation giving the action of the Galois group $\Gal(\F_{\ell}(\tilde E_p)/\F_{\ell})$ on $\tilde E_p$ via a choice of a basis. 
Let 
$\sigma_{\ell}$ be the Frobenius element of $\Gal(\F_{\ell}(\tilde E_p)/\F_{\ell})$. The order of 
$$\rho_{\tilde E,p}(\sigma_{\ell})\in \GL_2(\F_p)$$ is equal to $d$. 
Let $f_E\in \F_p[X]$ be its characteristic polynomial, which is given by
$$f_E=X^2-a_EX+\ell \pmod p.$$

1) Suppose $\Delta_E\not\equiv 0 \pmod p$. 

If $\Delta_E$  is a square in $\F_p^*$, the roots $\alpha, \beta$ of $f_E$ belong to $\F_p^*$ and are distinct. So up to conjugation, one has
$$\rho_{\tilde E,p}(\sigma_{\ell})=\begin{pmatrix} \alpha & 0 \\0 & \beta\end{pmatrix},$$
which implies $d=n$ as stated. 

If $\Delta_E$  is not a square in $\F_p^*$, then the roots of $f_E$ are $\alpha$ and $\alpha^p$ and belong to $\F_{p^2}$. So there exists a matrix $U\in \GL_2(\F_{p^2})$ such that 
$$\rho_{\tilde E,p}(\sigma_{\ell})=U\begin{pmatrix} \alpha & 0 \\0 & \alpha^p\end{pmatrix}U^{-1},$$
which leads  again $d=n$.

2) Suppose $\Delta_E\equiv 0 \pmod p$.  
We have
$f_E=(X-\alpha)^2$ so, up to conjugation, we obtain 
$$\rho_{\tilde E,p}(\sigma_{\ell})=\begin{pmatrix} \alpha & 0 \\0 & \alpha\end{pmatrix}\quad \hbox{or}\quad \rho_{\tilde E,p}(\sigma_{\ell})=\begin{pmatrix} \alpha & 1 \\0 & \alpha\end{pmatrix}.$$

From Lemma~\ref{L:lemma5}, we know $E/M$ has good ordinary reduction. 
Moreover, \cite[Corollary~2]{Freitaskraus} implies that 
$p$ divides $d$ if and only if $b_E\not\equiv 0 \pmod p$. 
This completes the proof of Theorem~\ref{T:thm1}.

We now prove Corollary~\ref{C:cor1}. We have $a_E=0$, so $\Delta_E=-4\ell$ is not divisible by $p$. From Theorem~\ref{T:thm1}, we obtain $d=n$. Moreover, one has $\alpha^2=-\ell$, so 
$$\delta=\frac{n}{\gcd(n,2)}.$$
Necessarily the order of $\alpha$ or the order of $-\alpha$ is even, so $n$ is even. Thus $n=2\delta$, as desired.

\section{Proof of Theorem~\ref{T:thm2} and Theorem~\ref{T:thm3}}
Suppose that $E$ has multiplicative reduction or additive potentially multiplicative reduction; in particular, one has $v(j)<0$. 
Let us denote
$$L=\Q_{\ell}\left(\sqrt{-c_6}\right).$$
 The elliptic curve $E/\Q_{\ell}$ is isomorphic over $L$ to the Tate curve $\G_m/q^{\Z}$, with $q\in \Z_{\ell}$ the element defined by the equality (\cite[p. 443]{Silverman} and \cite[Lemme~1]{CaliKraus})
 \begin{equation}
 \label{E:Tatej}
 j=\frac{1}{q}+744+196884 q+\cdots.
 \end{equation}
Let $\varepsilon : \Gal\left(\overline{\Q_{\ell}}/\Q_{\ell}\right)\to \lbrace \pm 1\rbrace$ be the character associated to the extension~$L/\Q_{\ell}$. Note that $\varepsilon$ can be of order 1. The curves $E/\Q_{\ell}$ and $\G_m/q^{\Z}$ are related by the quadratic twist by~$\varepsilon$. 

Let $\chi_p : \Gal\left(\overline{\Q_{\ell}}/\Q_{\ell}\right)\to \F_p^*$ be the mod~$p$ cyclotomic character. 
The representation giving the action of  $\Gal\left(\overline{\Q_{\ell}}/\Q_{\ell}\right)$ on $E_p$ is of the shape
\begin{equation}
\label{(9.2)}
\varepsilon \otimes \begin{pmatrix} \chi_p & * \\0 & 1\end{pmatrix}.
\end{equation}

\begin{lemma} \label{L:lemma6} 
The element $q$ is a $p$-th power in $\Q_{\ell}$ if and only if such is the case for $j$. In particular, one has $\Q_{\ell}\left(\mu_p,q^{\frac{1}{p}}\right)=\Q_{\ell}\left(\mu_p,j^{\frac{1}{p}}\right)$.
\end{lemma}

\begin{proof} From the equality \eqref{E:Tatej}, one has
$$j=\frac{u}{q}\quad \text{with}\quad u=1+744q+196884 q^2+\cdots.$$
Since $v(q)>0$, we have $u\equiv 1 \pmod \ell$. The primes $\ell$ and $p$ being distinct, by Hensel's Lemma we conclude that $u$ is a $p$-th power in $\Q_{\ell}$ and the result follows.
\end{proof}

\begin{lemma}  \label{L:lemma7} We have
$$\Q_{\ell}(E_p)=\Q_{\ell}\left(\sqrt{-c_6},\mu_p,j^{\frac{1}{p}}\right).$$
\end{lemma}

\begin{proof} The $\Gal\left(\overline{\Q_{\ell}}/L\right)$-modules $E_p$ and $\left(\mu_p,q^{\frac{1}{p}}\right)q^{\Z}/q^{\Z}$ are isomorphic. So one has 
$$L(E_p)=L\left(\mu_p,q^{\frac{1}{p}}\right).$$
The inequalities
$$[L(E_p):\Q_{\ell}(E_p)]\leq 2\quad \text{and}\quad p\geq 3,$$
imply that $q$ is a $p$-th power in $\Q_{\ell}(E_p)$. Moreover, from \eqref{(9.2)} it follows that for all  element $\sigma\in \Gal\left(\overline{\Q_{\ell}}/\Q_{\ell}\right)$ fixing $\Q_{\ell}(E_p)$, one has $\varepsilon(\sigma)=1$, so   $\sqrt{-c_6}$ belongs  
to $ \Q_{\ell}(E_p)$.
We  obtain
$$L(E_p)=\Q_{\ell}\left(\sqrt{-c_6},\mu_p,q^{\frac{1}{p}}\right)\subseteq \Q_{\ell}(E_p)\subseteq L(E_p),$$
which leads to $\Q_{\ell}(E_p)=\Q_{\ell}\left(\sqrt{-c_6},\mu_p,q^{\frac{1}{p}}\right)$. The result now follows from Lemma~\ref{L:lemma6}.
\end{proof}

\begin{lemma} 
\label{L:lemma8} 
Let $a$ be an element of $\Q_{\ell}^*$.
\begin{itemize}
\item[1)] If $\ell\not\equiv 1 \pmod p$, then 
$$a\in \Q_{\ell}^p \    \Leftrightarrow \  v(a)\equiv 0 \pmod p.$$

\item[2)] If $\ell\equiv 1 \pmod p$, then
$$a\in \Q_{\ell}^p \    \Leftrightarrow \  v(a)\equiv 0 \pmod p\quad \text{and}\quad \left(\frac{a}{\ell^{v(a)}}\right)^{\frac{\ell-1}{p}}\equiv 1 \pmod \ell.$$
\end{itemize}
\end{lemma} 

\begin{proof} Let $\mu_{\ell-1}$ be the group of $(\ell-1)$-th roots of unity and $U^1$ be the group of units of $\Z_{\ell}$ which are congruent to $1$ modulo $\ell$. 
One has $\Q_{\ell}^*=\mu_{\ell-1}\times U^1\times \ell^{\Z}$.
Since  $p\neq \ell$,  it follows from Hensel's lemma that each unit in $U^1$ is  a $p$-th power in $\Q_{\ell}$. 

In case $\ell\not\equiv 1 \pmod p$, the map $x\mapsto x^p$ is an automorphism  of $\mu_{\ell-1}$, hence the first assertion. 

Suppose  $\ell\equiv 1 \pmod p$. There exists $\zeta\in \mu_{\ell-1}$   such that $ \frac{a}{\ell^{v(a)}}\equiv \zeta \pmod \ell$. Moreover,
$\zeta$ belongs to $\Q_{\ell}^p$ if and only if $\zeta^{\frac{\ell-1}{p}}=1$.  This implies the result.
\end{proof}
Let us recall the essential fact that
$$r=[\Q_{\ell}(\mu_p):\Q_{\ell}].$$
We can now complete the proofs of Theorems~\ref{T:thm2}~and~\ref{T:thm3}.

\subsection{Proof of Theorem~\ref{T:thm2}} 1) Suppose $\ell\geq 3$ and $\left(\frac{-c_6}{\ell}\right)=1$, or $\ell=2$ and $c_6\equiv 7 \pmod 8$. 

From these assumptions,  $-c_6$ is a square in $\Q_{\ell}$. 
So we have from Lemma~\ref{L:lemma7} that
\begin{equation}\Q_{\ell}(E_p)=\Q_{\ell}\left(\mu_p,j^{\frac{1}{p}}\right).
\label{(9.3)}
\end{equation}
If $j$ is a $p$-th power in $\Q_{\ell}$, it follows that $d=r$, otherwise one has $d=pr$. Moreover, if $\ell\equiv 1 \pmod p$, one has $r=1$. The assertions 1.1) and 1.2) of Theorem~\ref{T:thm2} are then a  consequence of Lemma~\ref{L:lemma8}.

2) Suppose $\ell\geq 3$ and $\left(\frac{-c_6}{\ell}\right)=-1$, or $\ell=2$ and $c_6\not\equiv 7 \pmod 8$. 

Suppose  $r$ is even. The elliptic curve $E/\Q_{\ell}$ having multiplicative reduction,   the extension  $L/\Q_{\ell}$ is unramified. Consequently, $\sqrt{-c_6}$ belongs to  $\Q_{\ell}(\mu_p)$, so the equality \eqref{(9.3)} is again satisfied and the assertion 2.1) follows by Lemma~\ref{L:lemma8}.
Suppose $r$ is odd. Since $-c_6$ is not a square in $\Q_{\ell}$, it follows that  $-c_6$ is not a square in $\Q_{\ell}(\mu_p)$. Now assertion 2.2) follows from  Lemmas~\ref{L:lemma7}~and~\ref{L:lemma8}. This completes the proof 
of Theorem~\ref{T:thm2}. 

\subsection{Proof of Theorem~\ref{T:thm3}} 
Since  $E/\Q_{\ell}$ has additive reduction, the extension $L/\Q_{\ell}$ is ramified. Therefore, if $j$ is a $p$-th power in $\Q_{\ell}$ one has $d=2r$, otherwise $d=2pr$. Lemma~\ref{L:lemma8} now implies Theorem~\ref{T:thm3}.

\section{Proof of Lemmas~\ref{L:lemma1}, \ref{L:lemma2}, \ref{L:lemma3}, \ref{L:lemma4}}

\subsection{Proof of Lemma~\ref{L:lemma1}} We have $\ell\geq 5$  and $e=2$, so $v(\Delta)=6$.  The change of variables
 \[
  \begin{cases}
x=\ell X \\
y=\ell\sqrt{\ell} Y,\
  \end{cases}
\]
is an isomorphism from $E/\Q_{\ell}$ to its quadratic twist by $\sqrt{\ell}$, which is given by the equation
$$Y^2=X^3-\frac{c_4}{48\ell^2}X-\frac{c_6}{864\ell^3}.$$
It is an integral model whose discriminant is a unit of $\Z_{\ell}$, hence the lemma.

\subsection{Proof of Lemma~\ref{L:lemma2}} Recall that
$$u=\ell^{{v(\Delta)\over 12}}\quad \hbox{and}\quad  M=\Q_{\ell}(u).$$
The change of variables
\begin{equation}
\label{(10.1)} 
  \begin{cases}
x=u^2X \\
y=u^3Y,\
  \end{cases}
\end{equation}
is an isomorphism between  the elliptic $E/\Q_{\ell}$ given by the equation \eqref{(5.2)} and the 
elliptic curve $E'/M$ given by the equation \eqref{(5.5)}. The equation \eqref{(5.5)} is integral and the valuation of its  discriminant is
$$v(\Delta)-12v(u)=0,$$
which proves the lemma.

\subsection{Proof of Lemma~\ref{L:lemma3}}  The change of variables 
 \[
  \begin{cases}
X=3x \\
Y=3\sqrt{3}y,\
  \end{cases}
\]
realizes an isomorphism between the elliptic curve $E/\Q_3$ given by the equation \eqref{(6.1)} and its twis by $\sqrt{3}$ (denoted $E'/\Q_3$) of
equation
\begin{equation*}
(W') \; : \; Y^2=X^3-\frac{3c_4}{16}X-\frac{c_6}{32}.
\end{equation*} 
Let $c_4(W')$, $c_6(W')$ and $\Delta(W')$ be the standard invariants  associated to the model $(W')$. From the standatd invariats for $E/\Q_3$ we conclude
$$\left(v(c_4(W')),v(c_6(W')),v(\Delta(W'))\right)\in \lbrace  (4,6,12), (5,\geq 9,12)   \rbrace.$$
Moreover, the model $(W')$ is not minimal by \cite[p. 126, Table II]{Papadopoulos}; 
thus, $E'/\Q_3$ has good reduction over $\Q_3$, hence the lemma.

\subsection{Proof of Lemma~\ref{L:lemma4}} We adopt here 
the notations used in \cite{Papadopoulos}.

1) Suppose $t=(\geq 6,6,6)$ and $c_6'\equiv 1 \pmod 4$. 

An integral  model  of $E'/\Q_2$, the quadratic twist of $E/\Q_2$ by $\sqrt{2}$, is
$$(W')Ê\ : \ Y^2=X^3-\frac{c_4}{12}X-\frac{c_6}{108}.$$
It satisfies
$$\left(v(c_4(W'),v(c_6(W'),v(\Delta(W')\right)=(\geq 8,9,12).$$
We shall prove that $(W')$ is not minimal, which implies that $E'/\Q_2$ has good reduction. For this, we use the Table~IV and Proposition~6 of \cite{Papadopoulos}.
We have
$$b_8(W')=-a_4(W')^2,$$
hence the congruence
$$b_8(W')\equiv 0 \pmod {2^8}.$$
So we can choose $r=0$ in Proposition~6 of {\it{loc. cit.}}. Moreover, one has 
$$b_6(W')=4a_6(W')=-\frac{c_6}{27}.$$
Since $c_6'\equiv 1 \pmod 4$, one has $-\frac{c_6'}{27}\equiv 1 \pmod 4$. The equality $v(c_6)=6$ then implies 
$$b_6(W') \equiv 2^6 \pmod {2^8},$$
and we obtain  our assertion with $x=8$.

2) Suppose $t=(\geq 6,6,6)$ and $c_6'\equiv -1 \pmod 4$. 

We proceed as above. An equation of $E'/\Q_2$, the quadratic twist of $E/\Q_2$ by $\sqrt{-2}$,  is 
$$(W') \ : \ Y^2=X^3-\frac{c_4}{12}X+\frac{c_6}{108}.$$
One has again
$b_6(W')\equiv 2^6 \pmod {2^8},$
hence   the result. 

\medskip
For the next two cases below, we will denote by  $b_2$, $b_4$, $b_6$ and $b_8$ the standard invariants 
associated to  the  equation \eqref{(7.2)} of $E/\Q_2$.
\medskip

3) Suppose $t=(4,6,12)$.

An equation of $E'/\Q_2$, the quadratic twist of $E/\Q_2$ by $\sqrt{-1}$,  is 
$$(W') \ : \ Y^2=X^3-\frac{c_4}{48}X+\frac{c_6}{864}.$$
We  will use Table IV and Proposition 4 of \cite{Papadopoulos} to prove that
$(W')$ is not minimal, establishing that $E'/\Q_2$ has good reduction.

From the assumption made on  $t$, the elliptic curve $E/\Q_2$ corresponds to the  case 7 of Tate. One has $b_2=0$, $v(b_4)=1$, $v(b_6)=3$ and $v(b_8)=0$.
So there exists $r\in \Z_2$, with $v(r)=0$,  such that (conditions (a) and (b) of Proposition~4)
$$b_8+3rb_6+3r^2b_4+3r^4\equiv 0 \pmod {32}\quad \text{and}\quad r\equiv 1 \pmod 4.$$
Furthermore,
$$b_2(W')=0,\quad b_4(W')=b_4,\quad b_6(W')=-b_6,\quad b_8(W')=b_8.$$
We conclude that the integer $-r$ satisfies the condition (a) of 
the same Proposition 4 for the equation $(W')$.
One has $-3r\equiv 1 \pmod 4$, so condition (b) of this proposition with $s=1$  implies the assertion.

4) Suppose $t=(\geq 8,9,12)$.

Again an equation of $E'/\Q_2$  is
$$(W') \ : \ Y^2=X^3-\frac{c_4}{48}X+\frac{c_6}{864}.$$
We use Proposition~6 of \cite{Papadopoulos}. The elliptic curve $E/\Q_2$ corresponds to the  case 10 of Tate.  One has
$b_8\equiv 0 \pmod {2^8},$
so $r=0$ satisfies the required condition of this proposition for the equation \eqref{(7.2)}. Since equation \eqref{(7.2)} is  minimal, we deduce (by \cite[Prop~6]{Papadopoulos}) that  $b_6$ is not a square modulo~$2^8$, so we have
$$b_6=-\frac{c_6}{216}=-\frac{2^6 c_6'}{27}\equiv 2^6c_6'\equiv -2^6 \pmod {2^8},$$
where the last congruence follows due to $c_6'\equiv -1 \pmod 4$. 
From the equality $b_6(W')=-b_6$ it follows that $b_6(W')$ is a square modulo~$2^8$, hence the $(W')$ is not minimal, as desired.

5) Suppose $t=(6,9,18)$. 

Let $\varepsilon=\pm 1$, so that $c_6'\equiv \varepsilon \pmod 4$. 
An integral equation of  the quadratic twist of $E/\Q_2$ by $\sqrt{-2\varepsilon}$,  is 
$$(W') \ : \ Y^2=X^3-\frac{c_4}{2^6\cdot 3}X+\varepsilon\frac{c_6}{2^8\cdot 27}.$$
It satisfies
$$\left(v(c_4(W')),v(c_6(W')),v(\Delta(W'))\right)=(4,6,12).$$
We will apply Proposition 4 of \cite{Papadopoulos}.
From the assumption on~$t$, we have
$c_4'^3\equiv c_6'^2\pmod {32}$,
which implies
$$c_4'\equiv 1,9,17,25 \pmod {32}.$$
Moreover, one has $\varepsilon c_6'\equiv 1 \pmod 4$, so from the condition (a) for $(W')$, there exists $r\in \Z_2$ such that 
$$-c_4'^2+8r -18c_4'r^2+27r^4\equiv 0 \pmod {32}.$$
For all the values of $c_4'$ modulo $32$, we then verify that this congruence
is satisfied with $r=-1$. 
The condition (b) then implies that $(W')$ is not minimal, so  $E'/\Q_2$ has good reduction, hence  the lemma.

\section{Proof of Theorems~\ref{T:thm4},  \ref{T:thm9}, \ref{T:thm11}}

Let $\ell \geq 2$ and $E/\Q_\ell$ satisfy $e(E)=2$. 
Let $u\in \lbrace \pm 2,-1\rbrace$ be as defined in Lemma~\ref{L:lemma4}
and denote
 \[
w= 
  \begin{cases}
\ell \ \text{ if }   \ell\geq 3, \\
u \ \text{ if } \ell=2.\
  \end{cases}
\]
From lemmas 1--4, the quadratic twist $E'/\Q_{\ell} $ of $E/\Q_{\ell}$ by $\sqrt{w}$ has good reduction. Let $\varepsilon : \Gal\left(\Q_{\ell}(\sqrt{w})/\Q_{\ell}\right)\to \lbrace \pm 1\rbrace$ be the character associated to 
$\Q_{\ell}({\sqrt{w}})/\Q_{\ell}$.  In suitable basis of $E_p$ and $E_p'$, the representations
$$\rho_{E',p} : \Gal\left(\overline{\Q_{\ell}}/\Q_{\ell}\right)\to \GL_2(\F_p)\quad \text{\and}\quad \rho_{E,p} : \Gal\left(\overline{\Q_{\ell}}/\Q_{\ell}\right)\to \GL_2(\F_p),$$
giving the action of  $\Gal\left(\overline{\Q_{\ell}}/\Q_{\ell}\right)$ on $E_p'$ and $E_p$ satisfy the equality
\begin{equation}
\label{(11.1)} 
\rho_{E,p}=\varepsilon \cdot \rho_{E',p}.
\end{equation}

Let $H_0$ be the image of $\rho_{E',p}$.
From the criterion of N\'eron-Ogg-Shafarevitch, since $\ell \neq p$ and $E'/\Q_{\ell}$ has good reduction, the extension $\Q_{\ell}(E_p')/\Q_{\ell}$ is unramified (see \cite[p. 201, Thm~7.1]{Silverman}). So $H_0$ is cyclic. Let $\sigma_{\ell}\in \Gal\left(\overline{\Q_{\ell}}/\Q_{\ell}\right)$ be a lift  of the 
Frobenius element of the Galois group $\Gal\left(\Q_{\ell}(E_p')/\Q_{\ell}\right)$. Then $H_0$ is generated by
$$h_0=\rho_{E',p}(\sigma_{\ell}).$$

\begin{lemma} \label{L:lemma9}
Let $H$ be the subgroup of $\GL_2(\F_p)$ generated by $-1$ and $h_0$. Then, $H$ is the image of $ \rho_{E,p}$. In particular, one has
 \[
d= 
  \begin{cases}
|H_0| \ \text{ if }   -1\in H_0, \\
2|H_0|  \ \text{ otherwise}.\
  \end{cases}
\]
\end{lemma}

\begin{proof}  The equality \eqref{(11.1)} implies that  the image of $\rho_{E,p}$ is contained in $H$. 

Conversely,  by assumption the inertia subgroup of $\Gal\left(\Q_{\ell}(E_p)/\Q_{\ell}\right)$ has order $2$. 
Noting   $\tau \in \Gal\left(\overline{\Q_{\ell}}/\Q_{\ell}\right)$  a lift  of its  generator, $\rho_{E,p}(\tau)$ belongs to $\SL_2(\F_p)$, so one has  
\begin{equation}
\label{(11.2)} 
\rho_{E,p}(\tau)=-1.
\end{equation}

Moreover, the extension $\Q_{\ell}({\sqrt{w}})/\Q_{\ell}$ being ramified, it  is not contained in $\Q_{\ell}(E_p')$. So the restriction map induces an isomorphism 
between  
$$\Gal\left(\Q_{\ell}({\sqrt{w}})(E_p)/\Q_{\ell}({\sqrt{w}})\right)\quad \text{and}\quad   \Gal(\Q_{\ell}(E_p')/\Q_{\ell}).$$
We deduce  there exists $\sigma\in \Gal\left(\overline{\Q_{\ell}}/\Q_{\ell}\right)$ such that  $\sigma\left(\sqrt{w}\right)=\sqrt{w}$ and the restriction of $\sigma$ and $\sigma_{\ell}$ to $\Q_{\ell}(E_p')$ are equal.
From  \eqref{(11.1)}, we obtain 
$\rho_{E,p}(\sigma)=h_0$
which shows that $H$ is contained in the image of $\rho_{E,p}$, hence the lemma.
\end{proof}
Recall that the  characteristic polynomial in $\F_p[X]$ of $h_0$ is
$$f_{E'}=X^2-a_{E'}X+\ell \pmod p.$$

\subsection{Proof of the assertion 1 in theorems~\ref{T:thm4}, \ref{T:thm9}, \ref{T:thm11}} We assume $\Delta_{E'}\not\equiv 0 \pmod p$. Up to conjugation by a matrix in $\GL_2(\F_{p^2})$, we have
$$h_0=\begin{pmatrix} \alpha & 0 \\0 & \beta\end{pmatrix},$$
so $|H_0|=n$. Recall that $d=[\Q_{\ell}(E_p):\Q_{\ell}]$. 
It follows from Lemma~\ref{L:lemma9} that 
 \[
d= 
  \begin{cases}
n \ \text{ if }   -1\in H_0, \\
2n  \ \text{ otherwise}.\
  \end{cases}
\]
We will prove that $-1$ belongs to $H_0$ if and only if $n$ is even and $\alpha^{\frac{n}{2}}=\beta^{\frac{n}{2}}=-1$. 
This will establish the assertion 1 of theorems.

Suppose $-1$ belongs to $H_0$. Then  there exists $k\in \lbrace 1,\cdots,n\rbrace$, such that
 \begin{equation}
 \label{(10.4)} 
 \begin{pmatrix} \alpha^k & 0 \\0 & \beta^k\end{pmatrix}=-1.
 \end{equation}
One has $\alpha^{2k}=\beta^{2k}=1$, so $n$ divides $2k$.
Since $k\leq n$, one has $2k=n$ or $k=n$. Since $\alpha^n=\beta^n=1$ and $p\neq 2$, we have $k\neq n$, so $n$ is even and we obtain $\alpha^{\frac{n}{2}}=\beta^{\frac{n}{2}}=-1$. 
Conversely, if  $n$ even and $\alpha^{\frac{n}{2}}=\beta^{\frac{n}{2}}=-1$, we have
$-1=h_0^{\frac{n}{2}}\in H_0$, as desired.

\subsection{Proof of the assertion 2 in theorems~\ref{T:thm4}, \ref{T:thm9}, \ref{T:thm11}} We assume $\Delta_{E'}\equiv 0 \pmod p$.  

The elliptic curve $E'/\Q_{\ell}$ has good ordinary reduction by Lemma~\ref{L:lemma5}. 
Moreover, the polynomial~$f_{E'}$ has a single root
$$\alpha=\frac{a_{E'}}{2} \pmod p.$$
From \cite[Theorem~2]{Centeleghe} there exists a suitable basis of $E_p'$ in which
 $$h_0 = \begin{pmatrix} \alpha & 0 \\  b_{E'}  &  \alpha\end{pmatrix}.$$
We conclude that 
 \[
|H_0|= 
  \begin{cases}
n \ \text{ if }   b_{E'}\equiv 0 \pmod p, \\
np  \ \text{ otherwise}.\
  \end{cases}
\]
For all integers $k \geq 1$, one has
  $$h_0^k=\begin{pmatrix} \alpha^k & 0 \\   kb_{E'}\alpha^{k-1}   &  \alpha^k\end{pmatrix}.$$

\begin{lemma} \label{L:lemma10}
One  has
$-1\in H_0$ if and only if $n$ is even and  $\alpha^{\frac{n}{2}}=-1$.
\end{lemma}

\begin{proof} Suppose $-1\in H_0$ i.e. there exists an integer $k$ such  that $-1=h_0^k$. Then 
$$\alpha^k =-1\quad \text{and}\quad kb_{E'}\alpha^{k-1}=0.$$
Since   $p\neq 2$, $n$ being the order of $\alpha$ in $\F_p^*$,  by considering the value of $k$ modulo $p$, we have inequalities
$$1\leq k\leq n\leq p-1.$$
Moreover, $\alpha^{2k} =1$, so $n$ divides $2k$. Then $2k=n$ or $2k=2n$, which leads to $n=2k$, hence the implication.

Conversely, suppose $n$ even and $\alpha^{\frac{n}{2}}=-1$. One has  $\alpha^{\frac{pn}{2}}=-1$, which implies $h_0^{\frac{pn}{2}}=-1\in H_0$ and  proves the lemma.
\end{proof}

The assertions 2.1) and 2.2) of Theorem~\ref{T:thm4} are now a direct 
consequence of lemmas~\ref{L:lemma9}~and~\ref{L:lemma10}.

Suppose  $\ell=2$.  Since $E'/\Q_2$ has good ordinary reduction,  we have $a_{E'}=\pm 1$, so $\Delta_{E'}=-7$ and we obtain $p=7$. Moreover, $b_{E'}^2$ divides $\Delta_{E'}$, so $b_{E'}=1$.
If $a_{E'}=1$ one has $n=3$ and if $a_{E'}=-1$ one has $n=6$.  In both cases, lemmas~\ref{L:lemma9}~and~\ref{L:lemma10} imply $d=42$ as stated.

Suppose  $\ell=3$.  One has $a_{E'}\in \lbrace \pm 1,\pm 2\rbrace$, so $\Delta_{E'}\in \lbrace -11,-8\rbrace$.  Since $p$ divides $\Delta_{E'}$, this implies $a_{E'}=\pm 1$  and $\Delta_{E'}=-11$. In particular, $p=11$.
One has again $b_{E'}=1$. Furthermore, if $a_{E'}=1$ one has $n=10$ and if $a_{E'}=-1$ one has $n=5$. This leads to $d=110$ 
(by lemmas~\ref{L:lemma9}~and~\ref{L:lemma10}).

This completes the proofs of theorems~\ref{T:thm4}, \ref{T:thm9}, \ref{T:thm11}.

\section{Notation for the proof 
of theorems~\ref{T:thm5}, \ref{T:thm6}, \ref{T:thm7}, \ref{T:thm8}}

We have $\ell\geq 5$ and the elliptic curve $E/\Q_{\ell}$ has additive potentially good reduction, with a semistability defect $e\in\lbrace 3,4,6\rbrace$. 
Let $M$ and the elliptic curve $E'/M$ be defined as in \eqref{(5.4)} and  \eqref{(5.5)}. We will write
$$K=\Q_{\ell}(E_p),\quad G=\Gal(K/\Q_{\ell}), \quad G'=\Gal(M(E_p')/M)\quad \text{and}\quad  d_0=[M(E_p'):M].$$
The elliptic curve $E'/M$ has good reduction (Lemma~\ref{L:lemma2}), so the extension $M(E_p')/M$ is unramified and cyclic. Since $M/\Q_{\ell}$ is totally ramified, the value of $d_0$ can be determined from Theorem~\ref{T:thm1}. 

Let $\Frob_M\in G'$
be the Frobenius element of $G'$. Recall that $\Frob_M$ is a generator of $G'$. We have $M(E_p')=MK$ and the extension $K/\Q_{\ell}$ is Galois, so 
the Galois groups $\Gal(K/M\cap K)$ and $G'$ are isomorphic via the restriction morphism. Let 
$\Frob_K\in G$
be the restriction of $\Frob_M$ to $K$. It  is a generator of 
$\Gal(K/M\cap K)$. In particular, $\Frob_K$ and $\Frob_M$ have the same order.

Recall that the inertia subgroup of $G$ is cyclic of order $e$. 
We let $\tau$ be one of its generators. 

Let $\calB$ be a basis of $E_p$. Let $\calB'$ be the basis of $E_p'$ which is the image of $\calB$ by the isomorphism from $E_p$ to $E_p'$ given by  the change of variables \eqref{(10.1)}. Denote by
$$\rho_{E,p} : G\to \GL_2(\F_p)\quad \text{and}\quad \rho_{E',p} : G' \to \GL_2(\F_p)$$
the faithful representations giving the actions of $G$ and $G'$ on $E_p$ and $E_p'$ in the basis $\calB$ and $\calB'$, respectively. We have 
\begin{equation}
\label{(12.1)} 
\rho_{E,p}(\Frob_K)=\rho_{E',p}(\Frob_M). 
\end{equation}
Since the determinant of 
$\rho_{E,p}$ is the mod~$p$ cyclotomic character, we also have 
\begin{equation}
\label{(12.2)}
\rho_{E,p}(\tau)\in \SL_2(\F_p).
\end{equation}
Finally, let also
$$\sigma_{E,p} : G \to \PGL_2(\overline{\F_p})\quad \text{and}\quad \sigma_{E',p} : G'\to \PGL_2(\overline{\F_p})$$
be the associated projective representations extended to $\overline{\F_p}$. 

We write $F$ for the fixed field by the kernel of $\sigma_{E,p}$.  
 
\section{Proof of Theorem~\ref{T:thm5} and assertion 1 of Theorem~\ref{T:thm8}}

In this case, we have $e=[M:\Q_{\ell}]=3$, so either $M \subset K$ or $M\cap K=\Q_{\ell}$. 

\begin{lemma} \label{L:lemma11}
The degree $d=[K : \Q_\ell]$ satisfies
 \[
d= 
  \begin{cases}
3d_0 \ \text{ if }   M\subseteq K, \\
d_0  \ \text{ otherwise}.\
  \end{cases}
\]
\end{lemma}

\begin{proof} Suppose $M$ is contained in $K$. Since $E$ and $E'$ are isomorphic over $M$ (Lemma~\ref{L:lemma2}),   we have $K=M(E_p')$. Therefore, $d_0=[K:M]$ and
$$d=[K:M][M:\Q_{\ell}]=3d_0.$$
In case $M$ is not contained in $K$, the Galois groups $G$ and $G'$ are isomorphic, so $d=d_0$, hence the lemma.
\end{proof}

\subsection{Assertion 1) of Theorem~\ref{T:thm8}} The field $M$ is contained in $K$ (\cite[Lemma~5]{Freitaskraus}), 
so $d=3d_0$ by Lemma~\ref{L:lemma11}.
From the assumption $\ell\equiv 2 \pmod 3$ and \cite[Lemme~1]{Kraus2} it 
follows that $E'/M$ has good supersingular reduction. 
Since $\ell\geq 5$ we must have $a_{E'}=0$.  
Now by Corollary~\ref{C:cor1} we obtain $d_0=2\delta$ and 
assertion 1 of Theorem~\ref{T:thm8} follows.

\subsection{Assertion 1) of Theorem~\ref{T:thm5}} 
Assume $p\neq 3$. 

Since $3$ divides $\ell-1$, the  group $G$ is abelian by \cite[Corollary~3]{Freitaskraus}.
So there exists 
$U\in \GL_2(\F_{p^2})$ and $a\in \F_{p^2}^*$ such that  
\begin{equation}
\label{(13.1)}
U\rho_{E,p}(\Frob_K)U^{-1}=\begin{pmatrix} \alpha & a \\0 & \beta\end{pmatrix}\quad \text{and}\quad U\rho_{E,p}(\tau)U^{-1}=\begin{pmatrix} \zeta & 0 \\0 & \zeta^2\end{pmatrix},
\end{equation}

\begin{lemma} \label{L:lemma12}
We have $a=0$. In particular, $d_0=n$.
\end{lemma}

\begin{proof} 
Using the fact that   $\Frob_K$ and $\tau$ commute, we obtain the equality $a(1-\zeta)=0$, so $a=0$. We deduce that $\Frob_K$ has order $n$. Moreover, $\Frob_K$ and $\Frob_M$ have the same order and the latter has order $d_0$.
\end{proof}

\begin{lemma} \label{L:lemma13}
We have that $M \not\subset K$ if and only if 
\begin{equation}
\label{(13.2)} 
n\equiv 0 \pmod 3\quad \text{and}\quad  \lbrace \alpha^{\frac{n}{3}},\beta^{\frac{n}{3}}\rbrace=\lbrace \zeta,\zeta^2\rbrace.
\end{equation}
If this condition is satisfied then $d=n$, otherwise  $d=3n$.
\end{lemma}

\begin{proof} Suppose that $M$ is not contained in $K$.  Then  $G$ and $G'$ are  isomorphic, so  $G$  is cyclic generated by $\Frob_K$.
Consequently, there exist $k\in \lbrace 1,\cdots,n\rbrace$ such that 
$$\alpha^k=\zeta \quad \text{and}\quad \beta^k=\zeta^2.$$
We have $k\neq n$ and  $\alpha^{3k}=\beta^{3k}=1$, so $n$ divides $3k$. Since $3k\leq 3n$, this  implies  $3k\in \lbrace n, 2n \rbrace$, so $3$ divides $n$ and $k=\frac{n}{3}$ or $k=\frac{2n}{3}$. 
If $\alpha^{\frac{2n}{3}}=\zeta$, then $\alpha^{\frac{4n}{3}}=\alpha^{\frac{n}{3}} =\zeta^2$ and similarly $\beta^{\frac{n}{3}} =\zeta$, hence the condition \eqref{(13.2)}.

Conversely, suppose $M \subset K$. We shall prove that, for all $k\geq 1$, we have
\begin{equation}
\label{(13.3)}
\lbrace \alpha^k,\beta^k\rbrace\neq \lbrace \zeta, \zeta^2\rbrace.
\end{equation}
This implies \eqref{(13.2)} is not satisfied, completing the proof of the first statement.

From our assumption, 
one has $K=M(E_p')$ and $\Frob_K$ is a generator of $\Gal(K/M)$. The fixed field by $\tau$ is an unramified extension of $\Q_{\ell}$. 
In particular,  $\tau$ does not fix $M$. We deduce that for all $k\geq 1$, one has
$\Frob_K^k\neq \tau$, which implies  \eqref{(13.3)}.

The last statement now follows from lemmas~\ref{L:lemma11} and \ref{L:lemma12}.
\end{proof}

\begin{lemma} \label{L:lemma14}   
Suppose that $p \mid \Delta_{E'}$. Then $M \subset K$ and $d=3n$.
\end{lemma}

\begin{proof}  From our assumption, one has $\alpha=\beta$. Suppose that $M$ is not contained in $K$. Then   $\Frob_K$ is a generator of $G$. The eigenvalues of $\rho_{E,p}(\tau)$ are distinct, 
so the condition \eqref{(13.1)} leads to a contradiction. The result follows from  
lemmas~\ref{L:lemma11} and \ref{L:lemma12}.
\end{proof}
This completes the proof of the first assertion  of Theorem~\ref{T:thm5}.

\subsection{Assertion 2) of Theorem~\ref{T:thm5}.} Suppose $p=3$.

There exist $U\in \GL_2(\F_9)$ and $a\in \F_9$ such that 
$U\rho_{E,p}(\Frob_K)U^{-1}=\begin{pmatrix} \alpha & a \\0 & \beta\end{pmatrix}.$
Moreover,   the image of $\rho_{E,3}$ being a subgroup of $\GL_2(\F_3)$, one has 
$d\not\equiv 0 \pmod 9.$

Suppose $M$ is not contained in $K$. The group $G$ and $G'$  are  isomorphic, so $G$ is generated by $\Frob_K$. Since $3$ divides $d$ (because $e=3$), the order of $a\in \F_9$ is equal to $3$, so  $d=3n$.

Suppose $M$ is contained in $K$. Then one has $d=3[K:M]$. One  deduces that  $3$ does not divide $[K:M]$. The group
$\Gal(K/M)$ being generated by $\Frob_K$, this implies  $a=0$, so $[K:M]=n$ and  $d=3n$ as stated. 

This completes the proof of Theorem~\ref{T:thm5}.

\section{Proof of Theorem~\ref{T:thm6}}

In this case, we have $e=[M:\Q_{\ell}]=4$. Recall that $\zeta_4$ is a primitive $4$th root of unity.

The extension $M/\Q_{\ell}$ is cyclic, so $M\cap K$ is $\Q_{\ell}$, $\Q_{\ell}(\sqrt{\ell})$ or $M$. 

\begin{lemma} \label{L:lemma15}
The degree $d=[K : \Q_\ell]$ satisfies
 \[
d= 
  \begin{cases}
4d_0 \ \text{ if }   M\subseteq K, \\
2d_0  \ \text{if}  \ M\cap K=\Q_{\ell}(\sqrt{\ell}),\\
d_0 \ \text{ if }  M\cap K=\Q_{\ell}.\
  \end{cases}
\]
\end{lemma}

\begin{proof} The Galois groups $G'$ and $\Gal(K/M\cap K)$ are isomorphic and $|G'|=d_0$, hence the result.
\end{proof}
Since $4$ divides $\ell-1$, 
the Galois group $G$ is abelian by \cite[Cor. 3]{Freitaskraus}. 
So 
 there exist $U\in \GL_2(\F_{p^2})$  and $a\in \F_{p^2}$ such that  
\begin{equation}
\label{(14.1)} 
U\rho_{E,p}(\Frob_K)U^{-1}=\begin{pmatrix} \alpha & a \\0 & \beta\end{pmatrix}\quad \text{and}\quad U\rho_{E,p}(\tau)U^{-1}=\begin{pmatrix} \zeta_4 & 0 \\0 & \zeta_4^{-1}\end{pmatrix}.
\end{equation}

\begin{lemma}  \label{L:lemma16}
We have $a=0$. In particular, $d_0=n$.
\end{lemma}

\begin{proof} Using the fact that   $\Frob_K$ and $\tau$ commute, we obtain the equality $a(\zeta_4-\zeta_4^{-1})=0$, so $a=0$. Moreover, the orders of $\Frob_K$ and $\Frob_M$ are equal, so $d_0$ is the order of  $\Frob_K$,
hence the assertion.
\end{proof}

\begin{lemma} \label{L:lemma17} 
We have $M\cap K=\Q_{\ell}$  if and only if
\begin{equation}
\label{(14.2)} 
n\equiv 0 \pmod 4\quad \text{and}\quad  \lbrace \alpha^{\frac{n}{4}},\beta^{\frac{n}{4}}\rbrace=\lbrace \zeta_4,\zeta_4^{-1}\rbrace.
\end{equation}
If this condition is fulfilled, then $d=n$.
\end{lemma}

\begin{proof}
Suppose $M\cap K=\Q_{\ell}$. Then, $G$ are  $G'$  isomorphic and $\Frob_K$ generates $G$. From     \eqref{(14.1)} and Lemma~\ref{L:lemma16}, we deduce that  there exists $k\in \lbrace 1,\cdots,n\rbrace$ such that 
$$\lbrace \alpha^k,\beta^k\rbrace=\lbrace \zeta_4,\zeta_4^{-1}\rbrace.$$
One has $\alpha^{4k}=\beta^{4k}=1$, so $n$ divides $4k$. One has $k\neq n$. Moreover, if $4k=2n$, then $k=\frac{n}{2}$ so $\alpha^n=\beta^n=-1$ which is not. As in the proof of Lemma~\ref{L:lemma13}, 
we conclude $n=4k$ or $3n=4k$, which implies \eqref{(14.2)}. 

Conversely, suppose the condition \eqref{(14.2)} is satisfied. We deduce that $\tau$ belongs to the subgroup of  $G$ generated by $\Frob_K$. 
Furthermore, the fixed field by $\tau$ is an unramified extension of $\Q_{\ell}$. In particular, the extension $M\cap K/\Q_{\ell}$ must be unramified, hence
 $M\cap K=\Q_{\ell}$. 
 
Now lemmas~\ref{L:lemma15}~and~\ref{L:lemma16} imply $d=n$, as desired. 
\end{proof}

\begin{lemma} \label{L:lemma18} 
The field $M$ is contained in $K$ if and only if
\begin{equation}
\label{(14.3)} 
n\equiv 1 \pmod 2\quad \text{or}\quad  \lbrace \alpha^{\frac{n}{2}},\beta^{\frac{n}{2}}\rbrace\neq \lbrace-1\rbrace.
\end{equation}
If this condition is fulfilled, then $d=4n$.
\end{lemma}

\begin{proof} 
The order of $\tau^2$ is equal to $2$. From  \eqref{(12.2)}, one has
$$\rho_{E,p}(\tau^2)=-1.$$

Suppose $M$ is contained in $K$. 
Then, $\tau^2$ does not fix $M$, because the extension $K/M$ is unramified. Since $\Gal(K/M)$ is generated by $\Frob_K$,  for all $k\in \lbrace 1,\cdots,n\rbrace$, one has $\Frob_K^k\neq -1$.
The condition \eqref{(14.1)} then implies \eqref{(14.3)}.
 
Conversely, suppose that the condition \eqref{(14.3)} is satisfied. From Lemma~\ref{L:lemma17}, we conclude that $\Q_{\ell}(\sqrt{\ell})$ is contained in $K$. So  $\tau^2$ belongs  to  $\Gal(K/\Q_{\ell}(\sqrt{\ell}))$.
 Suppose $M\cap K=\Q_{\ell}(\sqrt{\ell})$. Then the Galois groups $\Gal(K/\Q_{\ell}(\sqrt{\ell}))$ and $G'$ are isomorphic. Consequently, 
 $\Gal(K/\Q_{\ell}(\sqrt{\ell}))$ is generated by $\Frob_K$ and there  exists $k\in \lbrace 1,\cdots,n\rbrace$ such that $\Frob_K^k=\tau^2$. This means that $\alpha^k=\beta^k=-1$, which implies
 $n$ even and $\alpha^{\frac{n}{2}}=\beta^{\frac{n}{2}}=-1$. This  leads to a contradiction, hence $M$ is contained in $K$. 
 
 If  \eqref{(14.3)} is satisfied, we then obtain $d=4n$ from lemmas~\ref{L:lemma15} and~\ref{L:lemma16}.
\end{proof}

\begin{lemma} \label{L:lemma19} 
We have $M\cap K=\Q_{\ell}(\sqrt{\ell})$  if and only if the conditions \eqref{(14.2)} and \eqref{(14.3)} are not satisfied. In such case, we have $d=2n$.
\end{lemma}
\begin{proof} It follows directly from the previous lemmas.
\end{proof}

This completes the proof of part 1) of Theorem~\ref{T:thm6}.

\begin{lemma} \label{L:lemma20} 
Suppose $p \mid \Delta_{E'}$. Then $\Q_{\ell}(\sqrt{\ell})$ is contained in $K$. Moreover, we have $d=2n$ if and only if the condition \eqref{(14.3)} is not fulfilled; otherwise, $d=4n$.
\end{lemma}

\begin{proof} Under this assumption, one has $\alpha=\beta$. Suppose  $\Q_{\ell}(\sqrt{\ell})$ is not   contained in $K$. Then $G'$ and $\Gal(K/\Q_{\ell})$ are isomorphic, so $\tau$ is a power of $\Frob_K$. The condition \eqref{(14.1)} implies a contradiction because the eigenvalues of $\rho_{E,p}(\tau)$ are distinct.

Consequently, $d=4n$ or $d=2n$ by Lemma~\ref{L:lemma15}. The result now follows 
from lemmas~\ref{L:lemma18},~\ref{L:lemma19}.
\end{proof}

Finally, part 2) of Theorem~\ref{T:thm6} follows from Lemma~\ref{L:lemma20}.
 
\section{Proof of Theorem~\ref{T:thm7}}

In this case, we have $e=[M:\Q_{\ell}]=6$. Recall that $\zeta_6$ is a primitive $6$-th root of unity.

The extension $M/\Q_{\ell}$ is cyclic, so $M\cap K$ is $\Q_{\ell}$, $\Q_{\ell}(\sqrt{\ell})$, $\Q_{\ell}(\root{3}\of \ell)$ or $M$. 

Recall that $d_0=[M(E_p'):M]$. 

\begin{lemma} \label{L:lemma21}
The degree $d=[K : \Q_\ell]$ satisfies
 \[
d= 
  \begin{cases}
6d_0 \ \text{ if }   M\subseteq K, \\
3d_0  \ \text{if}  \ M\cap K=\Q_{\ell}(\root{3}\of \ell),\\
2d_0  \ \text{if}  \ M\cap K=\Q_{\ell}(\sqrt{\ell}),\\
d_0 \ \text{ if }  M\cap K=\Q_{\ell}.\
  \end{cases}
\]
\end{lemma}

\begin{proof} The Galois groups $G$ and $\Gal(K/M\cap K)$ are   isomorphic, hence the result.
\end{proof}

\subsection{Proof of part 1)} Assume $p\neq 3$. 

After taking the quadratic twist of $E/\Q_\ell$ by $\sqrt{\ell}$, the semistablity defect is of order $3$. 
Moreover, taking quadratic twist does not change the fact that the image of $\rho_{E,p}$ is abelian or not.
So, as in the case $e=3$, we conclude that $G$ is abelian (by \cite[Cor. 3]{Freitaskraus}). 
Consequently, there exist $U\in \GL_2(\F_{p^2})$  and $a\in \F_{p^2}$ such that  
\begin{equation}
\label{(15.1)} 
U\rho_{E,p}(\Frob_K)U^{-1}=\begin{pmatrix} \alpha & a \\0 & \beta\end{pmatrix}\quad \text{and}\quad U\rho_{E,p}(\tau)U^{-1}=\begin{pmatrix} \zeta_6 & 0 \\0 & \zeta_6^{-1}\end{pmatrix}.
\end{equation}

\begin{lemma} \label{L:lemma22}
We have $a=0$. In particular, $d_0=n$.
\end{lemma}
\begin{proof} This follows as Lemma~\ref{L:lemma12}.
\end{proof}

\begin{lemma} \label{L:lemma23}  
We have $M\cap K=\Q_{\ell}$  if and only if
\begin{equation}
\label{(15.2)} 
n\equiv 0 \pmod 6\quad \text{and}\quad  \lbrace \alpha^{\frac{n}{6}},\beta^{\frac{n}{6}}\rbrace=\lbrace \zeta_6,\zeta_6^{-1}\rbrace.
\end{equation}
If this condition is fulfilled, then $d=n$.
\end{lemma}

\begin{proof}
Suppose $M\cap K=\Q_{\ell}$. Then, $G$ are  $G'$  isomorphic and $\Frob_K$ generates $G$.
 From  \eqref{(15.1)} and Lemma~\ref{L:lemma23} there exists $k\in \lbrace 1,\cdots,n\rbrace$ such that 
$\lbrace \alpha^k,\beta^k\rbrace=\lbrace \zeta_6,\zeta_6^{-1}\rbrace.$
We have $\alpha^{6k}=\beta^{6k}=1$, so $n$ divides $6k$. 
As in the proof of Lemma~\ref{L:lemma17}, we conclude that \eqref{(15.2)} holds.
 
Conversely, if \eqref{(15.2)} is satisfied, $\tau$ belongs to the subgroup of  $G$ generated by $\Frob_K$. 
The fixed field by $\tau$  being an unramified extension of $\Q_{\ell}$, this implies $M\cap K=\Q_{\ell}$. 
 
Finally, when \eqref{(15.2)} is satisfied, we then obtain $d=n$ from lemmas~\ref{L:lemma21} and~\ref{L:lemma22}.
\end{proof}

\begin{lemma}\label{L:lemma24} 
We have $M\cap K=\Q_{\ell}(\sqrt{\ell})$  if and only the two following conditions are satisfied~:
\begin{itemize}
\item[1)] $n\not\equiv 0 \pmod 6$ or $\lbrace \alpha^{\frac{n}{6}},\beta^{\frac{n}{6}}\rbrace\neq \lbrace \zeta_6,\zeta_6^{-1}\rbrace$.
\medskip
\item[2)] $n\equiv 0 \pmod 3$ and  $\lbrace \alpha^{\frac{n}{3}},\beta^{\frac{n}{3}}\rbrace=\lbrace \zeta_6^2,\zeta_6^{-2}\rbrace$.
\end{itemize}

Moreover, if these conditions are fulfilled, then $d=2n$.

\end{lemma} 

\begin{proof} Suppose $M\cap K=\Q_{\ell}(\sqrt{\ell})$. From lemma 24 the first condition is satisfied.
The order of $\tau^2$ is equal to $3$, hence $\tau^2$ belongs to $\Gal(K/\Q_{\ell}(\sqrt{\ell}))$, which is isomorphic to $G'$. So there exists $k\in \lbrace 1,\cdots,n\rbrace$ such that $\tau^2=\Frob_K^k$ 
and  \eqref{(15.1)} leads to $\lbrace \alpha^k,\beta^k \rbrace=\lbrace \zeta_6^2,\zeta_6^{-2}\rbrace$. One has $\alpha^{3k}=\beta^{3k}=1$, so $n$ divides $3k$, and since $n\neq k$, this  implies the second condition.

Conversely, suppose the two conditions of the statement are satisfied. Then $M\cap \Q_{\ell}\neq \Q_{\ell}$ By condition 1 and Lemma~\ref{L:lemma24}. The second condition implies that $\tau^2$ belongs to
the subgroup $\Gal(K/M\cap K)$ of $G$ generated by $\Frob_K$. In particular, the ramification index of the extension $K/M\cap K$ is at least $3$, hence $M\cap K=\Q_{\ell}(\sqrt{\ell})$.

The last statement follows from lemmas~\ref{L:lemma22}~and~\ref{L:lemma23}.
\end{proof}

\begin{lemma} \label{L:lemma25}
We have $M\cap K=\Q_{\ell}(\root{3} \of \ell)$ if and only the following conditions are satisfied~:
\begin{itemize}
\item[1)] $n\not\equiv 0 \pmod 6$ or $\lbrace \alpha^{\frac{n}{6}},\beta^{\frac{n}{6}}\rbrace\neq \lbrace \zeta_6,\zeta_6^{-1}\rbrace$.
\medskip
\item[2)] $n\equiv 0 \pmod 2$ and  $\alpha^{\frac{n}{2}}=\beta^{\frac{n}{2}}=-1$.
\end{itemize}
Moreover, if these conditions are fulfilled, then $d=3n$.
\end{lemma} 

\begin{proof}
Suppose $M\cap K=\Q_{\ell}(\root{3} \of \ell)$. The first condition is satisfied (lemma 24).
The order of $\tau^3$ is equal to $2$, so $\tau^3$ belongs to $\Gal(K/\Q_{\ell}(\root{3} \of \ell))$, which is isomorphic to $G'$. So there exists $k\in \lbrace 1,\cdots,n\rbrace$ such that $\tau^3=\Frob_K^k$ 
and  \eqref{(15.1)} leads to $\lbrace \alpha^k,\beta^k \rbrace=\lbrace \zeta_6^3,\zeta_6^{-3}\rbrace$. One has $\alpha^{2k}=\beta^{2k}=1$ hence  the second condition.

Conversely, one has  $M\cap \Q_{\ell}\neq \Q_{\ell}$ (condition 1 and lemma 24). Moreover, $3$ does not divide $n$. Otherwise,  from the second condition, it  follows that the first one 
is not   satisfied.
We deduce that  $M\cap \Q_{\ell}\neq \Q_{\ell}(\sqrt{\ell})$ (lemma 25).
Moreover, the second condition implies that $\tau^3$ belongs to
the  $\Gal(K/M\cap K)$. So the ramification index of the extension $K/M\cap K$ is at least $2$, hence $M\cap K=\Q_{\ell}(\root{3} \of \ell)$ and the lemma.

The last statement follows from lemmas~\ref{L:lemma22} and~\ref{L:lemma23}.
\end{proof}

The assertion 1.1) of Theorem~\ref{T:thm7} is now a consequence of the previous lemmas. 

\begin{lemma} \label{L:lemma26}
Suppose that $p \mid \Delta_{E'}$. 
Then $\Q_{\ell}(\root{3} \of \ell)$ is contained in $K$. One has $d=3n$ if and only if 
the two conditions of Lemma~\ref{L:lemma25} are satisfied; otherwise, $d=6n$.
\end{lemma}

\begin{proof} Note that $p \mid \Delta_{E'}$ implies $\alpha=\beta$.
Suppose  $\Q_{\ell}(\root{3} \of \ell)$   is not   contained in $K$. Then $G'$ is either isomorphic to $\Gal(K/\Q_{\ell})$ or $\Gal(K/\Q_{\ell}(\sqrt{\ell}))$.
So $\tau$ or $\tau^2$ is a power of $\Frob_K$. The condition \eqref{(15.1)} implies a contradiction because the eigenvalues of $\rho_{E,p}(\tau)$ and $\rho_{E,p}(\tau^2)$ are distinct.

Consequently, one has $d=3n$ or $d=6n$ by lemma~\ref{L:lemma21}; 
Lemma~\ref{L:lemma25} now completes the proof.
\end{proof}

We can now conclude part 1.2) as follows. Since $M\cap \Q_{\ell}\neq \Q_{\ell}$ (Lemma~\ref{L:lemma26}), it follows from Lemma~\ref{L:lemma24} that condition 1 of Lemma~\ref{L:lemma25} is satisfied. So if $n$ is even and $\alpha^{\frac{n}{2}}=-1$, then $d=3n$ (Lemma~\ref{L:lemma26} since $\alpha = \beta$). 
Otherwise, $d=6n$ (by Lemma~\ref{L:lemma26}), as desired.

\subsection{Assertion 2} \label{S:16.2}
The group $G$ is abelian (\cite[Cor. 3]{Freitaskraus}).  
Let $\Phi$ be the inertia subgroup of~$G$.
Up to conjugation, there exists only one cyclic subgroup of order $6$ in $\GL_2(\F_3)$.
 So we can suppose that $\rho_{E,p}(\Phi)$  is generated by the matrix 
$\begin{pmatrix} 2 & 2 \\0 & 2\end{pmatrix}$. We deduce that the normalizer  of  $\rho_{E,p}(\Phi)$ is the standard Borel subgroup $B$ of $\GL_2(\F_3)$. Since $\Phi$ is a normal subgroup of $G$,  this implies that 
$\rho_{E,p}(G)$ is contained in $B$. 
The group $B$ is non-abelian of order $12$. Since $6$ divides $|\rho_{E,p}(G)|$, one deduces that $|\rho_{E,p}(G)|=6$, hence $d=6$.

\section{Proof of assertion  2 of Theorem~\ref{T:thm8}} Recall that one has $e\in \lbrace 4,6\rbrace$. Since $\ell\equiv -1 \pmod e$, the group $G$ is not abelian (\cite[Cor. 3]{Freitaskraus}).  
Let $\Phi$ be the inertia subgroup of $G$. 

\subsection{Case $(p,e)=(3,6)$} As in section~\ref{S:16.2}, we can suppose that $\rho_{E,p}(G)$  is contained in the standard Borel subgroup $B$ of $\GL_2(\F_3)$, which is of order $12$. Since $G$ is not abelian and $\Phi$ is cyclic of order $6$, this implies   $\rho_{E,p}(G)=B$, hence $d=12$. Moreover, one has $\ell\equiv 2 \pmod 3$, so $r=2$ and we obtain $d=er$, as desired.

\subsection{Case $p\geq 5$ or $(p,e)=(3,4)$} \label{S:thm8e4}
We will use for our proof the results established in \cite{Diamond} and \cite{DiamondKramer}
and we will adopt the notations and terminology  used in these papers.

The group of the $e$-th roots of unity is not contained in 
$\Q_{\ell}$. Since 
$p\geq 5$  or $(p,e)=(3,4)$, we deduce 
from \cite[Proposition 0.3]{DiamondKramer}, that the representation  $\sigma_{E,p} : G \to \PGL_2(\overline{\F_p})$ is of type {\bf{V}}. Denote by $\Phi$ the inertia subgroup of $G$. 
Recall that $F$ is the field fixed by the kernel of $\sigma_{E,p}$ and write $H=\Gal(K/F)$ and $d'=[K:F]$. 

Recall also that $r=[\Q_{\ell}(\mu_p):\Q_{\ell}]$.

For all $\sigma\in H$, $\rho_{E,p}(\sigma)$ is a scalar matrix in $\GL_2(\overline{\F_p})$. So there exists a character $\varphi : H\to \overline{\F_p}^*$ such that 
$$\rho_{E,p}|_H\simeq \begin{pmatrix} \varphi  & 0 \\ 
0 &\varphi \end{pmatrix}.$$

The order of $\varphi$ is $d'$. Moreover, $\chi_p|_H : H\to \Aut(\mu_p)$ being the cyclotomic character giving the action of $H$ on $\mu_p$, 
one has
$$\varphi^2=\chi_p|_H.$$

 We then deduce the equality
\begin{equation}
\label{(16.1)}
\frac{d'}{\gcd(d',2)}=[F(\mu_p):F].
\end{equation}
Moreover, since $\Phi$ is cyclic of order $e$, one has
\begin{equation}
\label{(16.2)} 
|\sigma_{E,p}(\Phi)|=\frac{e}{2}.
\end{equation}
Furthermore, the assumption $\ell\equiv -1 \pmod e$ implies that $\Q_{\ell}(\zeta_e)$ is the  quadratic unramifed extension of $\Q_{\ell}$.

\begin{lemma} \label{L:lemma27}
We have $F=\Q_{\ell}(\zeta_e,\ell^{\frac{e}{2}})$. In particular, $d=ed'$.
\end{lemma}

\begin{proof} The Galois group $\Gal(F/\Q_{\ell})$ is isomorphic to the dihedral group of order $e$ 
(see condition 3 in \cite[Proposition~2.3]{Diamond} and \eqref{(16.2)}). 
The group $\Phi$ fixes an unramified extension of~$\Q_{\ell}$. 
From \eqref{(16.2)} it follows the extension $F/\Q_{\ell}$ is not totally ramified, hence 
the result in case $e=4$. If $e=6$, since $\ell\equiv 2 \pmod 3$, then \cite[Theorem 7.2]{Roblot} 
implies the result.
\end{proof}

1) Suppose $r$ even. Then $\zeta_e$ belongs to $\Q_{\ell}(\mu_p)$. Lemma~\ref{L:lemma27} 
implies $F\cap \Q_{\ell}(\mu_p)=\Q_{\ell}(\zeta_e)$. 
So  the Galois groups $\Gal(F(\mu_p)/F)$ and $\Gal(\Q_{\ell}(\mu_p)/\Q_{\ell}(\zeta_e))$ 
are  isomorphic, and we obtain 
$$[F(\mu_p):F]=\frac{r}{2}.$$
The ramification index of $K/F$ is equal to $\frac{e}{2}$, so $2$ divides $d'$. From \eqref{(16.1)} we then deduce $d'=r$, which leads to  $d=er$ by Lemma~\ref{L:lemma27}, as desired.

2) Suppose $r$ odd. Then $\zeta_e$ does not belong to $\Q_{\ell}(\mu_p)$.
 This implies $F\cap \Q_{\ell}(\mu_p)=\Q_{\ell}$, so $[F(\mu_p):F]=r$, i.e. the order of $\chi_p|_H$ is $r$. 
 We deduce that $\varphi^{2r}=1$, so $d'$ divides $2r$.
Since $F(\mu_p)$ is contained in $K$, $r$ divides $d'$. So we have $d'=r$ or $d'=2r$. 
One has  $d'\neq r$, otherwise $K=F(\mu_p)$ and this contradicts the fact that  the ramification index of $K/\Q_{\ell}$ is $e$ and  the one of $F(\mu_p)/\Q_{\ell}$ is $\frac{e}{2}$. 
So $d'=2r$ and we obtain $d=2er$, hence the result.

This completes the proof of Theorem~\ref{T:thm8}.

\section{Proof of Theorem~\ref{T:thm10}} 
Let $E/\Q_3$ be an elliptic curve with additive potentially good reduction with $e \neq 2$.

Let $\rho_{E,p} : \Gal(\overline{\Q_3}/\Q_3) \rightarrow \GL_2(\F_p)$ 
be the representation arising on the $p$-torsion points of $E$.

Write $K=\Q_3(E_p)$ for the field fixed by the kernel of $\rho_{E,p}$ whose degree is $d=[K : \Q_3]$.

Let the projective representation obtained from $\rho_{E,p} \otimes \overline{\F}_p$ be denoted by
$$\sigma_{E,p} : \Gal(\overline{\Q_3}/\Q_3) \to \PGL_2(\overline{\F}_p)$$ 
and write $F$ for the field fixed by its kernel. 

A minimal equation of $E/\Q_3$ is (see \cite[p. 355]{Kraus})
\begin{equation}
\label{(17.1)}
y^2=x^3-\frac{c_4}{48}x-\frac{c_6}{864}.
\end{equation}

Let $\Delta^{\frac{1}{4}}$ be a $4$th root of $\Delta$ in $\overline{\Q_3}$  and $E_2$ be the group of $2$-torsion points of $E(\overline{\Q_3})$. From  the corollary in \cite[p. 362]{Kraus}, the smallest extension of $\Q_3^{\unr}$ over which  $E/\Q_3^{\unr}$ acquires good reduction is
\begin{equation}
\label{(17.2)}
L=\Q_3^{\unr}\left(E_2,\Delta^{\frac{1}{4}}\right).
\end{equation}
In particular, we have
\begin{equation}
\label{(17.3)}
[L:\Q_3^{\unr}]=e.
\end{equation}

We will divide the proof according to the value of $e \in \{ 3,4,6,12 \}$.

\subsection{Case $e=3$}  We have (cf. {\it{loc.cit.}}, cor. p. 355)
$$\bigl(v(c_4),v(c_6),v(\Delta)\bigr)\in \lbrace (2,3,4), (5,8,12)\rbrace.$$
Let $x_0$ be the $x$-coordinate of a point of $E_2$ in the model {\eqref{(17.1)}. Let us denote 
$$M=\Q_3(x_0).$$

\begin{lemma} \label{L:lemma28}
The extension $M/\Q_3$ is totally ramified of degree $3$ 
and  $E/M$ has  good  supersingular reduction.
\end{lemma}

\begin{proof} Since $4$ divides $v(\Delta)$, $\Delta$ is a $4$th power in $\Q_3^{\unr}$. So from \eqref{(17.2)} one has $L=\Q_3^{\unr}\left(E_2\right)$ and 
$E/L$ acquires good reduction. Since we have $e=3$,   we deduce from    \eqref{(17.3)}   the equality $[\Q_3^{\unr}(E_2):\Q_3^{\unr}]=3$. This implies   $\Q_3^{\unr}(E_2)=\Q_3^{\unr}(x_0)$. 
So the extension $M/\Q_3$ is totally ramified of degree $3$ and $E/M$ has good reduction.  One has $2v(c_6)\neq v(\Delta)$, so $E/M$ has good supersingular reduction by \cite[p. 21, Lemme 7]{Kraus2}. 
\end{proof}

Let $E'/M$ be an elliptic curve with good reduction isomorphic over $M$ to $E/M$ and 
$\tilde {E'}/\F_3$ be the elliptic curve obtained from $E'/M$ by reduction. 
Let also
$$a_{E'}=4-|\tilde {E'}(\F_3)|.$$

\begin{lemma}\label{L:lemma29} 
We have $a_{E'}=0$.
\end{lemma}

\begin{proof} The elliptic curve $E'/M$ has a rational point 
of order $2$ rational over $M$ (Lemma~\ref{L:lemma28}). 
It follows
$\tilde{E'}/\F_3$ has also a rational point of order $2$ rational over $\F_3$. Since $E'/M$ has  good  supersingular reduction, up to an $\F_3$-isomorphism,  an equation 
of $\tilde{E'}/\F_3$ is $Y^2=X^3-X$ or $Y^2=X^3+X$, which implies the result.
\end{proof}

1) Suppose  that the extension $K/\Q_3$ is non-abelian. 
The Galois group $\Gal(M(E_p)/M)$ being cyclic, implies that $M$ is contained in $K$. So $M(E'_p)=K$. From Corollary~\ref{C:cor1} and Lemma~\ref{L:lemma29}, one has $[K:M]=2\delta$, so $d=6\delta$ as stated.

2) Let us assume that  extension $K/\Q_3$ is abelian.  One has $p\neq 3$. Let $\zeta\in \F_9$ be a primitive cubic root of unity  and $\alpha\in \F_9$ such that $\alpha^2=-3$. Let $n$ the least common multiple of the orders of $\alpha$ and $-\alpha$ in $\F_9$. 

\begin{lemma} \label{L:lemma30}
We have $d=3n$.
\end{lemma}

\begin{proof} The same arguments of Lemma~\ref{L:lemma13} allow to conclude that
$M$ is not contained in $K$ if and only if one has 
\begin{equation}
\label{(17.4)} 
n\equiv 0 \pmod 3\quad \text{and}\quad  \lbrace \alpha^{\frac{n}{3}},(-\alpha)^{\frac{n}{3}}\rbrace=\lbrace \zeta,\zeta^2\rbrace.
\end{equation}
Moreover, if this condition is fulfilled one has $d=n$, otherwise  $d=3n$. If  $3$ divides $n$, one has $(-\alpha)^{\frac{n}{3}}=\pm \alpha^{\frac{n}{3}}$, but  $\zeta\neq \pm \zeta^2$. So  the condition \eqref{(17.4)} is not satisfied, hence the lemma. 
\end{proof}

\begin{lemma} \label{L:lemma31}
We have $n=2\delta$.
\end{lemma}
 
\begin{proof} Let $s$ be the order of $\alpha$. We have $s=\delta \gcd(2,s)$. If $s=2\delta$, we have $n=s$ and the assertion. If  $s=\delta$, then $s$ must be odd, so the order of $-\alpha$ is $2s$, hence $n=2\delta$.
\end{proof}

Lemmas~\ref{L:lemma30} and~\ref{L:lemma31} imply the first part of the theorem in case $K/\Q_3$ is abelian, 
which completes the proof of Theorem~\ref{T:thm10} for $e=3$.

\subsection{Case $e=4$} The same proof as in section~\ref{S:thm8e4} for $e=4$ gives the result.

\subsection{Case $e=6$} We have  (\cite[p. 355, Cor]{Kraus})
 $$\left(v(c_4),v(c_6),v(\Delta)\right)\in \lbrace (3,5,6), (4,6,10)\rbrace.$$
 
\begin{lemma} \label{L:lemma32}
The group $\Gal(K/\Q_3)$ is abelian if and only if one has $\frac{\Delta}{3^{v(\Delta)}}\equiv 1 \pmod 3$.
\end{lemma}

\begin{proof} Let  $E'/\Q_3$  be the quadratic twist of $E/\Q_3$ by $\sqrt{3}$. The elliptic curve $E'/\Q_3$  has additive potentially good reduction with a semi-stability defect of order $3$  (cf. {\it{loc.cit.}}).
Moreover,  the extension $\Q_3(E'_p)/\Q_3$ is abelian if and only if such is the case of $K/\Q_3$.
The result now follows from \cite[Proposition 5]{Freitaskraus}.
\end{proof}

\subsection{Case $e=6$ with $K/\Q_3$ is abelian} Let $x_0$ be the $x$-coordinate of a point of $E_2$ in the model {\eqref{(17.1)}.  Let us denote 
\begin{equation}
\label{(17.5)} 
M=\Q_3\left(x_0,\sqrt{3}\right).
\end{equation}
\begin{lemma} \label{L:lemma33}
Suppose $\Gal(K/\Q_3)$ abelian. Then the  extension $M/\Q_3$ is totally ramified of degree $6$ and  $E/M$ has  good  supersingular reduction.
\end{lemma}

\begin{proof}  
Since $2$ divides $v(\Delta)$, we deduce from Lemma~\ref{L:lemma32} that $\Delta$ is a square in $\Q_3$. Moreover, $\frac{\Delta}{3^{v(\Delta)}}$ is a $4$-th power in $\Q_3$, so one has
$L=\Q_3^{\unr}(E_2,\sqrt{3})$ (equality  \eqref{(17.2)}). Furthermore, $\Delta$ being a square in $\Q_3$, and $3$ dividing $[\Q_3^{\unr}(E_2):\Q_3^{\unr}]$, one has $\Q_3^{\unr}(E_2)=\Q_3^{\unr}(x_0)$. 
In particular, $L=\Q_3^{\unr}(x_0,\sqrt{3})$. 
Since $[L:\Q_3^{\unr}]=6$ (equality \eqref{(17.3)}), we deduce that $M/\Q_3$ is totally ramified of degree $6$ and $E/M$ has good reduction. The fact that  $2v(c_6)\neq v(\Delta)$ implies that the reduction 
is supersingular \cite[p. 21, lemme 7]{Kraus2}.
\end{proof}

Let $E'/M$ be an elliptic curve with good reduction isomorphic over $M$ to $E/M$ and 
$\tilde {E'}/\F_3$ be the elliptic curve obtained from $E'/M$ by reduction. Let  
$a_{E'}=4-|\tilde {E'}(\F_3)|.$
Using Lemma~\ref{L:lemma33}, the same proof as the one establishing Lemma~\ref{L:lemma29} 
leads to the equality $a_{E'}=0.$
Let  $\alpha\in \F_9$ such that $\alpha^2=-3$ and  $n$ be the least common multiple of the orders of $\alpha$ and $-\alpha$ in $\F_9$. Let $\zeta$ be a primitive $6$-th root of unity in $\F_9$.
Using exactly the same arguments,  lemmas 22-25 are also valid with $\ell=3$ and the field $M$ defined by \eqref{(17.5)}. Since one has 
$\zeta\neq \pm \zeta^{-1}$ and $\zeta^2\neq \pm \zeta^{-2}$, we get
 \[
d= 
  \begin{cases}
3n \ \text{ if }   n\equiv 0 \pmod 2 \ \text{and}\ \alpha^{\frac{n}{2}}=(-\alpha)^{\frac{n}{2}}=-1, \\
6n  \ \text{ otherwise}.\
  \end{cases}
\]

\begin{lemma} \label{L:lemma34}
We have
 \[
d= 
  \begin{cases}
6\delta \ \text{ if }   \delta\equiv 0 \pmod 2, \\
12\delta  \ \text{ if }   \delta\equiv 1 \pmod 2.\
  \end{cases}
\]
\end{lemma} 

\begin{proof} As in the proof of Lemma~\ref{L:lemma31}, we have $n=2\delta$.

Suppose $\delta$ even. Let $\delta=2\delta'$. One has 
$(-\alpha)^{\frac{n}{2}}=\alpha^{\frac{n}{2}}=\alpha^{2\delta'}=(-3)^{\delta'}=-1,$
hence $d=3n=6\delta$. If $\delta$ is odd, one has $(-\alpha)^{\frac{n}{2}}\neq \alpha^{\frac{n}{2}}$, so $d=6n=12\delta$, hence the result.
\end{proof} 

To end  this section,  let us give the link between $\delta$ and $r$ for any prime number $\ell$. We will use this lemma several times.

\begin{lemma} \label{L:lemma35}
We have
  \[
\delta= 
  \begin{cases}
r  \ \text{ if } r, \delta  \ \text{are both even}, \\
\frac{r}{2}  \ \text{ if } r   \ \text{is even}\ \text{and}\ \delta \ \text{is odd},\\
2r   \ \text{ if } r   \ \text{is odd}.\
  \end{cases}
\]
\end{lemma}

\begin{proof} If $r$ is odd, the equality $-\ell=(-1)\ell$ implies $\delta=2r$. 
If  both $\delta$ and $r$ are even, one has 
$\ell^r=(-\ell)^r=1$ and $\ell^{\delta}=(-\ell)^{\delta}=1$
so that $\delta \mid r$ and $r \mid \delta$,  thus $r=\delta$.
Suppose $r$ even and $\delta$ odd. One has
$(-\ell)^{\delta}=(-1)^{\delta}\ell^{\delta}=1,$
so $\ell^{2\delta}=1$, and $r$ divides $2\delta$. Moreover, 
$\ell^r=(-\ell)^r=1,$
so $\delta$ divides $r$. We obtain $r=\delta$ or $r=2\delta$. This leads to $r=2\delta$ because $r$ and $\delta$ have not the same parity, hence the result.
\end{proof}
Lemmas~\ref{L:lemma34} and \ref{L:lemma35} now imply Theorem~\ref{T:thm10} for $e=6$ and $K$ abelian.

\subsection{Case $e=6$ with $K/\Q_3$ is non-abelian} We have
$j-1728=\frac{c_6^2}{\Delta}.$ In particular, $v(j-1728)$ is even. From our assumption and Lemma~\ref{L:lemma32}, this implies that $j-1728$ is not a square in $\Q_3$. We conclude the representation
 $\sigma_{E,p} : G \to \PGL_2(\overline{\F_p})$ is of type {\bf{V}} (by \cite[Cor. 0.5]{DiamondKramer}). 
Since $\Phi$ is cyclic of order $6$, one has $|\sigma_{E,p}(\Phi)|=3$, so the extension $F/\Q_3$ is dihedral of degree $6$ (\cite{Diamond}, prop. 2.3). This extension is not totally ramified, so the unramified quadratic extension $\Q_3(\sqrt{2})$ of $\Q_3$ is contained in the field $F$ fixed by the kernel of $\sigma_{E,p}$. 
 
 The end of the proof is now the same as the one used in section~\ref{S:thm8e4}.  Let $d'=[K:F]$. One has $d=6d'$.
 In case $r$ is even, $\sqrt{2}$ belongs to $\Q_3(\mu_p)$, so $[F(\mu_p):F]=\frac{r}{2}$. 
The ramification index of $K/F$ is equal to $2$, consequently  $2$ divides $d'$. The equality  \eqref{(16.1)} then leads to $d'=r$, 
and $d=6r$ as stated.
If $r$ odd, then $\sqrt{2}$ does not belong to $\Q_3(\mu_p)$.
 This implies   $F\cap \Q_3(\mu_p)=\Q_3$ and  $[F(\mu_p):F]=r$. We then conclude that 
$d'=2r$ and  $d=12r$, hence the result.

\subsection{Case $e=12$} Recall that $K=\Q_3(E_p)$. Let
$$M=\Q_3\left(E_2,\Delta^{\frac{1}{4}}\right).$$
Recall that we have $\Q_3(E_2)=\Q_3\left(\sqrt{\Delta},x_0\right)$ with $x_0$ being    a root of the  $2$-division polynomial of $E/\Q_3$. 
In particular, $[M:\Q_3]\leq 12$. 

\begin{lemma} \label{L:lemma36}
The  extension $M/\Q_3$ is totally ramified of degree $12$ and  $E/M$ has  good  supersingular reduction.
\end{lemma}

\begin{proof} 
We have  $[L:\Q_3^{\unr}]=12$ (equality \eqref{(17.3)}), so  $M/\Q_3$ is totally ramified of degree $12$ and $E/M$ has good reduction. We have  $2v(c_6)\neq v(\Delta)$ (see \cite[cor. p. 355]{Kraus})  so  the reduction 
is supersingular (see \cite[p. 21, lemme 7]{Kraus2}).
\end{proof}

\begin{lemma} \label{L:lemma37}
We have $[K\cap M:\Q_3]\in \lbrace 6,12\rbrace$.
\end{lemma}
\begin{proof}  Let $E'/M$ be an elliptic curve with good reduction isomorphic over $M$ to $E/M$ and 
$\tilde {E'}/\F_3$ be the curve obtained from $E'/M$ by reduction. Let  
$a_{E'}=4-|\tilde {E'}(\F_3)|.$ The points of order $2$ of $E'/M$ are rational over $M$, so $4$ divides $a_{E'}$, and the Weil bound implies 
$a_{E'}=0$. From Corollary~\ref{C:cor1}, we conclude
$$[M(E_p):M]=2\delta.$$
From Lemma~\ref{L:lemma35} we have $\delta\in \lbrace \frac{r}{2}, r, 2r\rbrace$, so
$$d=[K\cap M:\Q_3] u \quad \text{with}\quad u\in \lbrace r, 2r, 4r\rbrace.$$
Moreover, $12r$ divides $d$, hence $[K\cap M:\Q_3]\neq 1,2,4$. 

Suppose   $[K\cap M:\Q_3]=3$. Then $K\cap M$ is  contained in $\Q_3(E_2)$, otherwise $M/K\cap M$ would be an extension of degree divisible by $3$, hence $9$ would divide $[M:\Q_3]$, a contradiction.  
Since $K/\Q_3$ is Galois, this implies that $K\cap M=\Q_3(E_2)$ which leads to a contradiction, hence the lemma.
\end{proof}

The representation  $\sigma_{E,p} : G \to \PGL_2(\overline{\F_p})$ is of type {\bf H} by \cite[Prop. 0.4]{DiamondKramer}.  From \cite[Proposition 2.4]{Diamond}, we see that $\sigma_{E,p}(\Phi)$ is dihedral of order $6$. Furthermore, $F$ being the fixed field by the kernel of $\sigma_{E,p}$, implies $[F:\Q_3]=6$ or $[F:\Q_3]=12$ ({\it{loc. cit.}}). 

Let $d'=[K:F]$. We have $d\in \lbrace 6d',12d'\rbrace$ and, furthermore, \eqref{(16.1)} is still true.

\begin{lemma} \label{L:lemma38}
If  $[F:\Q_3]=6$, then $d=12r$.
\end{lemma}

\begin{proof} Since $|\sigma_{E,p}(\Phi)|=6$, our assumption implies that the extension $F/\Q_3$ is totally ramified. We have $e=12$, so $2$ divides $d'$ and \eqref{(16.1)} leads to $d'=2r$, hence the lemma.
\end{proof}

\begin{lemma}\label{L:lemma39} 
If  $r$ is even, then $d=12r$. 
\end{lemma}
\begin{proof} From Lemma~\ref{L:lemma38}  we can suppose $[F:\Q_3]=12$. Since $r$ is even, the quadratic unramified extension of $\Q_3$ is contained in $\Q_3(\mu_p)$. It is also contained in $F$ because 
the ramification index of $F/\Q_3$ is $6$. So we have $[F\cap \Q_3(\mu_p):\Q_3]=2$, hence $[F(\mu_p):F]=\frac{r}{2}$. Using \eqref{(16.1)}, the ramification index of $K/F$ being equal to $2$, we conclude $d'=r$,
so $d=12r$, as desired.
\end{proof}

\begin{lemma} \label{L:lemma40} 
Suppose $r$ is odd. We have $[K\cap M:\Q_3]=6$ and $d=24r$. 
\end{lemma}
\begin{proof} Suppose $[K\cap M:\Q_3]=12$. 
Then $M \subset K$ , so $d=12(2\delta)=24\delta$ by Corollary~\ref{C:cor1}.
Since $r$ is odd, one has $\delta=2r$ (Lemma~\ref{L:lemma35}), so $d=48r$. Moreover, from Lemma~\ref{L:lemma38}, we deduce that $[F:\Q_3]=12$ and \eqref{(16.1)} implies $d'\leq 2r$, so $d\leq 24r$, hence a contradiction.
We so obtain $[K\cap M:\Q_3]=6$ (Lemma~\ref{L:lemma37}) and we deduce that $d=6(2\delta)=6(4r)=24r$ as stated.
\end{proof}

This concludes the proof of Theorem~\ref{T:thm10}.

\section{Proof of Theorem~\ref{T:thm12}.}

Let $E/\Q_2$ be an elliptic curve with additive potentially good reduction with $e \neq 2$.

Let $\rho_{E,p} : \Gal(\overline{\Q_2}/\Q_2) \rightarrow \GL_2(\F_p)$ 
be the representation arising on the $p$-torsion points of $E$.

Write $K=\Q_2(E_p)$ for the field fixed by the kernel of $\rho_{E,p}$ whose degree is $d=[K : \Q_2]$.

Let the projective representation obtained from $\rho_{E,p} \otimes \overline{\F}_p$ be denoted by
$$\sigma_{E,p} : \Gal(\overline{\Q_2}/\Q_2) \to \PGL_2(\overline{\F}_p)$$ 
and write $F$ for the field fixed by its kernel. 

We will write $\Phi$ for both the inertia subgroups of $\Gal(\overline{\Q_2}/\Q_2)$ and $G = \Gal(K/\Q_2)$.

\subsection{Case $e=3$} Suppose that $E/\Q_2$ satisfies $e=3$.
Let  $\pi \in \overline{\Q_2}$ be a cubic root of $2$. 
Write $M=\Q_2(\pi)$. 
The extension $M/\Q_2$ is  totally ramified of degree $3$. 

\begin{lemma} \label{L:lemma41}
The elliptic curve $E/\Q_2$ acquires good supersingular reduction over $M$.
\end{lemma}

\begin{proof} The field $\Q_2^{\unr}(\pi)$ is the unique extension of degree $3$ of $\Q_2^{\unr}$. Since $e=3$, the elliptic curve $E/M$ has good reduction. Since $v(j)>0$ (see \cite{Kraus}), so the $j$ invariant modulo $2$ is $0$,  which is supersingular in characteristic $2$.
\end{proof}

We have $\ell=2\equiv -1 \pmod e$. It follows from \cite[Corollary 3  and Lemma 5]{Freitaskraus}  
that $K/\Q_2$ in non-abelian and  $M$ is contained in $K$. We conclude that
\begin{equation}
\label{(18.1)} 
d=3[M(E_p):M].
\end{equation}
Let $E'/M$ be an elliptic curve with good reduction isomorphic over $M$ to $E/M$. 
Since $E'/M$ has supersingular reduction, its trace of Frobenius 
satisfies $a_{E'}\in \lbrace -2,0,2\rbrace$.  

\begin{lemma} \label{L:lemma42}
We have $a_{E'}=0$.
\end{lemma}
 \begin{proof} The residue field of $M$ is $\F_2$, so $a_{E'}=3-|\tilde{E'}(\F_2)|$, 
where $\tilde{E'}$ is the reduction of $E'$.  

We will show $E/M$ has a rational $3$-torsion point over $M$. Together with 
the Weil bound, this gives the desired result. Let 
$\chi_3 : \Gal(\overline{\Q_2}/\Q_2)\to \F_3^*$ be the mod~$3$ cyclotomic character.
From \cite[Proposition 8]{Freitaskraus}, 
the representation $\rho_{E,3}$ is of one of the following types
$$ \begin{pmatrix} 1  & * \\ 
0 & \chi_3 \end{pmatrix} \quad \text{or} \quad   \begin{pmatrix} \chi_3  & * \\ 
0 & 1 \end{pmatrix},$$
in particular its image is of order $6$. 
So the restriction of $\rho_3$ to $\Gal(\overline{\Q_2}/M)$ is of the shape
$$ \begin{pmatrix} 1  & 0 \\ 
0 & \chi_3 \end{pmatrix}\quad \text{or} \quad   \begin{pmatrix} \chi_3  & 0 \\ 
0 & 1 \end{pmatrix},$$
hence the result.
\end{proof}

From Lemma~\ref{L:lemma41} and Corollary~\ref{C:cor1} 
we conclude $[M(E_p):M]=2\delta$. Now equation \eqref{(18.1)} 
gives $d=6\delta$, as desired.

\subsection{Case $e=4$} Suppose that $E/\Q_2$ satisfies $e=4$.
We consider the $3$-torsion field
$$M=\Q_2(E_3).$$

\begin{lemma} \label{L:lemma43} 
We have $[M:\Q_2]=8$ and $E/M$ has good supersingular reduction.
\end{lemma}

\begin{proof} It folloes from $e=4$ and Table 1 and Proposition 4 of \cite{DD2008} that $[M:\Q_2]=8$.
Since $v(j)>0$ (see \cite{Kraus}), so the $j$ invariant modulo $2$ is $0$,  which is supersingular in characteristic~$2$.
\end{proof}

Let $E'/M$ be an  elliptic curve of good reduction  isomorphic over $M$ to $E/M$.
The residue field of $M$ is $\F_4$ and $\tilde{E'}/\F_4$ has
full $3$-torsion over $\F_4$. It follows from 
$a_{E'}=5-|\tilde{E'}(\F_4)|$ and the Weil bound that
$|\tilde{E'}(\F_4)|=9$ and $a_{E'}=-4$. Moreover, $\Delta_{E'} = a_{E'}^2 - 4\cdot 2^2 = 0$
so (by \cite[Theorem 2]{Centeleghe}) the Frobenius element $\Frob_M$ of the extension $M(E_p)/M$
is representable by the matrix $\begin{pmatrix} -2 & 0 \\ 
0 &-2 \end{pmatrix}$. It is the homothety of ratio $-2$. Consequently, one has
\begin{equation}
\label{(18.2)}
[M(E_p):M]=\delta.
\end{equation}
 
\begin{lemma} \label{L:lemma44} 
We have $K\cap M \neq \Q_2$.
\end{lemma}
\begin{proof} Suppose $K\cap M=\Q_2$. Then the Galois groups of $K/\Q_2$ and $M(E_p)/M$ are isomorphic.
Let $\sigma$ be a generator of the inertia subgroup of $\Gal(K/\Q_2)$. The element $\rho_{E,p}(\sigma)$ belongs to $\SL_2(\F_p)$ and is of order $4$. Consequently, there exists $U\in \GL_2(\F_{p^2})$ such that 
$$U\rho_{E,p}(\sigma)U^{-1}=\begin{pmatrix} \zeta_4  & 0 \\ 
0 &-\zeta_4 \end{pmatrix}.$$
The restriction of the Frobenius $\Frob_M\in \Gal(M(E_p)/M)$ to $K$ is the homothety of ratio $-2$ and it is a generator of  $\Gal(K/\Q_2)$.  So there exists $k=1,\cdots,\delta$ such that 
$$\begin{pmatrix} \zeta_4  & 0 \\ 
0 &-\zeta_4\end{pmatrix}=\begin{pmatrix} (-2)^k  & 0 \\ 
0 &(-2)^k \end{pmatrix},$$
which leads to a contradiction, hence the result.
\end{proof}

\begin{lemma} \label{L:lemma45} 
Suppose $K/\Q_2$ is abelian. Then  $\Gal\left(M/\Q_2\right)$ is cyclic of order $8$.
\end{lemma}

\begin{proof} From \cite[Proposition 6]{Freitaskraus} it follows that $M/\Q_2$ is abelian and  
\cite[Table 1]{DD2008} gives the conclusion, since $M/\Q_3$ is of degree $8$. 
\end{proof}

\begin{lemma} \label{L:lemma46} 
We have $\mu_3\subseteq K$.
\end{lemma}

\begin{proof} From Lemma~\ref{L:lemma44} we know $[K\cap M:\Q_2]\geq 2.$

1) Suppose $K/\Q_2$ is abelian. Then Lemma~\ref{L:lemma45} implies that $K\cap M/\Q_2$ contains the quadratic unramified extension (the ramification degree of $M$ is $e=4$), as desired. 

2) Suppose  $K/\Q_2$ is non-abelian.  Because $\Phi$ is cyclic of order $4$, 
the representation $\rho_{E,p}$ is of type {\bf{V}} (by \cite[Proposition 2.3]{Diamond}).
Since $|\sigma_{E,p}(\Phi)|=2$, we conclude from condition 3 of 
\cite[Proposition 2.3]{Diamond} that the image of $\sigma_{E,p}$ is dihedral of order $4$.  
We have $F\subseteq K$ and $F/\Q_2$ is not totally ramified. 
Thus $\mu_3 \subset F \subset K$.
\end{proof}

\begin{lemma} \label{L:lemma47} 
The field $\Q_2(\mu_3)$ is strictly contained in $K\cap M$.  In particular,  
\begin{equation}
\label{(18.3)}
[K\cap M:\Q_2]\in \lbrace 4,8\rbrace.
\end{equation}
\end{lemma}

\begin{proof} Since $\mu_3 \subset M$ by Lemma~\ref{L:lemma46} we have $\mu_3 \subset K\cap M$. 
Suppose $K\cap M=\Q_2(\mu_3)$. Then $\mu_3$ is fixed by the generator $\sigma$ of the inertia subgroup of $\Gal\left(K/\Q_2\right)$. 
Up to conjugation, we have  $\rho_{E,p}(\sigma)=\begin{pmatrix} \zeta_4 & 0 \\ 
0 &-\zeta_4 \end{pmatrix}$.
The Galois groups $K/\Q_2(\mu_3)$ and $M(E_p)/M$ are isomorphic, so there exist 
(as in Lemma~\ref{L:lemma44}) an integer  $k=1,\cdots,\delta$ such that 
$$\begin{pmatrix} \zeta_4 & 0 \\ 
0 &-\zeta_4 \end{pmatrix}=\begin{pmatrix} (-2)^k  & 0 \\ 
0 &(-2)^k \end{pmatrix},$$
which gives a contradiction. Now Lemma~\ref{L:lemma43} implies \eqref{(18.3)}.
\end{proof}

\begin{lemma} \label{L:lemma48} 
There exists $k=1,\cdots,\delta$ such that 
$$ \begin{pmatrix} (-2)^k  & 0 \\ 
0 &(-2)^k \end{pmatrix}=-1$$
if and only if  $\delta$ is even. 
\end{lemma}

\begin{proof} If $\delta$ is even, one has $(-2)^{\frac{\delta}{2}}=-1$. Conversely, suppose there exists $k=1,\cdots,\delta$ such that $(-2)^k=-1$. One has $(-2)^{2k}=1$, so $\delta$ divides $2k$. If $\delta$ is odd, then $\delta$ divides $k$, so $k=\delta$ which is false, so $\delta$ is even, hence the lemma.
\end{proof}

\begin{lemma} \label{L:lemma49} 
We have the following equivalence
\begin{equation}
\label{(18.4)}
M\subseteq K \ \Leftrightarrow \  \delta\equiv 1 \pmod 2.
\end{equation}
\end{lemma}

\begin{proof} Let $\tau$ be the element of order $2$ of the inertia subgroup of  $\Gal(K/\Q_2)$. Its image $\rho_{E,p}(\tau)$ is in $\SL_2(\F_p)$ and is of order $2$, so 
\begin{equation}
\label{(18.5)} 
\rho_{E,p}(\tau)=-1.
\end{equation}
Suppose $M \not\subset K$. Then, Lemma~\ref{L:lemma47} 
implies both  $[K\cap M:\Q_2]=4$ and the  ramification index of the extension 
$K\cap M/\Q_2$ is equal to $2$. We conclude that $\tau \in \Gal(K/K\cap M)$.  
Since the Frobenius $\Frob_M\in \Gal(M(E_p)/M)$ restricts to a generator of $\Gal(K/K\cap M)$, 
there exists $k=1,\cdots,\delta$, such that $(-2)^k=-1$. Thus $\delta$ is even by Lemma~\ref{L:lemma48}.

Conversely, suppose  $M \subset K$; 
then $K/M$ is unramified. In particular, $\tau$ does not belong to 
$\Gal(K/M)$, which is generated by  $\Frob_M$ (because $M(E_p)=K$). 
From \eqref{(18.5)}, it follows there does not exist 
$k=1,\cdots,\delta$ such that $(-2)^k=-1$; thus $\delta$ is odd by Lemma~\ref{L:lemma48}.
\end{proof}

\begin{corollary} \label{C:cor2}
The degree $d=[K : \Q_2]$ satisfies
  \[
d= 
  \begin{cases}
4\delta \ \text{ if }  \delta \ \text{is even}, \\
8\delta   \ \text{ if }  \delta \ \text{is odd}.\
  \end{cases}
\]
\end{corollary}

\begin{proof} If $\delta$ is even, $M$ is not contained in $K$ (Lemma~\ref{L:lemma49}) 
so $[K\cap M:\Q_2]=4$ (Lemma~\ref{L:lemma47}). 
The equality \eqref{(18.2)} then implies $d=4\delta$. 

If $\delta$ is odd, $M$ is contained in $K$, so $d=8\delta$, as desired.
\end{proof}
\medskip

Corollary~\ref{C:cor2} and Lemma~\ref{L:lemma35} complete the proof 
of Theorem~\ref{T:thm12} for $e=4$.

\subsection{Case $e=6$} Suppose that $E/\Q_2$ satisfies $e=6$.
After a suitable quadratic twist we have $e=3$ 
and since $\ell = 2 \equiv -1 \pmod{3}$ we conclude 
that $K$ is non-abelian (\cite[Corollary~3]{Freitaskraus}). 
A similar argument to the case $e=6$ in the proof of
Theorem~\ref{T:thm8} part~2) gives the result.

\subsection{Case $e=8$} 
Suppose that $E/\Q_2$ satisfies $e=8$. We consider the fields
$$M=\Q_2(E_3), \qquad K=\Q_2(E_p), \qquad  L= K \cap M$$
which are three Galois extensions of $\Q_\ell$.

Since $e=8$, it is well known that $\Phi \subset \Gal(K/\Q_2)$ is isomorphic to quaternion group.

\begin{lemma} \label{L:lemma50} 
The image $\sigma_{E,p}(\Phi)$ is non-cyclic of order $4$ and the Galois group  $\Gal(F/\Q_2)$ is dihedral of order dividing $8$.
\end{lemma}

\begin{proof} Since $\Phi$ is isomorphic to quaternion group, its center $C$ is of order $2$ and $\Phi/C$ is isomorphic to the Klein group of order $4$. 
The result now follows from condition~3 of \cite[Proposition~2.4]{Diamond}.
\end{proof}

\begin{lemma}\label{L:lemma51} 
The degree $d=[K : \Q_2]$ satisfies  $d\in \lbrace 8r, 16r\rbrace$.
\end{lemma}
\begin{proof} We have $r=[\Q_2(\mu_p):\Q_2]$ and the unramified extension $\Q_2(\mu_p)/\Q_2$ is contained 
in~$K$; moreover, the inertia subgroup $\Phi \subset \Gal(K/\Q_2)$ is of order $8$. Then $8r \mid d$. 

Write $H=\Gal(K/F)$. There exists a character $\varphi : H\to \F_p^*$ such that 
\[ \rho_{E,p}|_H = \begin{pmatrix} \varphi  & 0 \\ 
0 &\varphi \end{pmatrix} \quad \text{ and } \quad \varphi^2=\chi_p|_{H}. \] 
The order of $\varphi$ is $d'=[K:F]$ and the order of $\chi_p|_{H}$ is 
$\frac{d'}{\gcd(d',2)}$ and it divides $r$. Hence $d'\leq 2r$ and from Lemma~\ref{L:lemma50}, 
we have $[F:\Q_2]\leq 8$. Thus $d\leq 16r$ and the lemma follows.
\end{proof}

\begin{lemma} \label{L:lemma52} 
The Galois group $\Gal(M/\Q_2)$ is semidihedral of order $16$.
\end{lemma} 

\begin{proof} This follows from Table~1 in \cite{DD2008}.
\end{proof}

\begin{lemma} \label{L:lemma53} 
We have $[L:\Q_2]\in\lbrace 4,8,16\rbrace$.
\end{lemma}

\begin{proof} The same argument leading to equality \eqref{(18.2)} shows that here we also have
\[
 [M(E_p):M]=\delta.
\]
Thus~$d=\delta [L:\Q_2]$.
From Lemma~\ref{L:lemma35} we conclude that $[L : \Q_2] = 4, 8, 16$ or $32$. Note
the last case is impossible due to Lemma~\ref{L:lemma52}, completing the proof.
\end{proof}

Let us denote by $SD_{16}$ the semidihedral group of order $16$.

\begin{corollary} \label{C:cor3}
We have $\mu_3\subseteq K$.
\end{corollary}

\begin{proof} This is clear if $[L:\Q_2]=16$ i.e. $L=M$ (because $\mu_3\subseteq M$). 

Suppose $[L:\Q_2]=8$. The only normal subgroup of $SD_{16}$ of order $2$ is its center and its quotient is dihedral. Thus $L/\Q_2$  is a dihedral extension of order $8$. Consequently, $L/\Q_2$ is not totally ramified, because $\Phi$ is the quaternion group, hence $\Q_2(\mu_3) \subset L$.

Suppose $[L:\Q_2]=4$. We will again show that $L/\Q_2$ is not totally ramified, which gives the result. 
The Galois group $\Gal(M/L)$ is a normal subgroup of order $4$ of $\Gal(M/\Q_2)$. There is only one normal subgroup of order $4$ of $SD_{16}$, and it is cyclic (its derived subgroup). 
Moreover,  $SD_{16}$ has exactly three cyclic subgroups of order $4$ and one subgroup isomorphic to $H_8$. The group $H_8$  has 
also  exactly three cyclic subgroups of order $4$. We deduce that the cyclic subgroups of order $4$ of $\Gal(M/\Q_2)$ are contained in its  inertia subgroup. 
This implies that $L/\Q_2$ is not totally ramified, hence the result.
\end{proof}

We can now complete the proof of Theorem~\ref{T:thm12} in the case $e=8$.

1) Suppose $r$ odd. Then $\mu_3 \not\subset \Q_2(\mu_p)$. Since $\Q_2(\mu_3)\subseteq K$ by Corollary~\ref{C:cor3}, we deduce that the degree of the maximal unramified subfield of $K$ is at least~$2r$. This implies $d\geq 16r$, so $d=16r$ by Lemma~\ref{L:lemma51}.
 
2) Suppose $r$ even. Then $\mu_3 \subset \Q_2(\mu_p)$. From Lemma~\ref{L:lemma50} 
we have that $F/\Q_2$ is dihedral with ramification index 4 and of degree $[F:\Q_2]=8$ 
or $[F:\Q_2]=4$. Write $d'=[K:F]$.
 
Suppose $[F:\Q_2]=8$. We have $\Q_2(\mu_p)\cap F=\Q_2(\mu_3)$. 
The Galois groups $\Gal(F(\mu_p)/F)$ and $\Gal(\Q_2(\mu_p)/\Q_2(\mu_p)\cap F)$ being isomorphic, 
we deduce that
$$[F(\mu_p):F]=\frac{r}{2}.$$
Using the notations in the proof of Lemma~\ref{L:lemma51}, we have that
the order of $\chi_p|_H$ is $\frac{r}{2}$, which leads to the equality
$$\frac{d'}{\gcd(d',2)}=\frac{r}{2}.$$
Since the ramification index of $K/F$ is $2$, we have $2 \mid d'$. 
Thus $d'=r$ and $d=8r$.
 
Suppose $[F:\Q_2]=4$. We have $$\frac{d'}{\gcd(d',2)}=[F(\mu_p):F]\leq r.$$
So $d'\leq 2r$, which implies $d\leq 8r$. Then $d=8r$ by Lemma~\ref{L:lemma51}, as desired.

\subsection{Case $e=24$}
Suppose that $E/\Q_2$ satisfies $e=24$. We consider the fields
$$M=\Q_2(E_3), \qquad K=\Q_2(E_p), \qquad  L= K \cap M$$
which are three Galois extensions of $\Q_\ell$.
Since $e=24$, it is well known that $\Phi \subset \Gal(K/\Q_2)$ is isomorphic to $\SL_2(\F_3)$ 
and $\Gal(M/\Q_2) \simeq \GL_2(\F_3)$.

\begin{lemma} \label{L:lemma54} 
The image $\sigma_{E,p}(\Phi)$ is of order $12$ isomorphic to $A_4$. Moreover, the Galois group  $\Gal(F/\Q_2)$ is isomorphic to $A_4$ or $S_4$.
\end{lemma}

\begin{proof} 
The only scalar matrices in $\SL_2(\F_p)$ are $\pm 1$, therefore
$\sigma_{E,p}(\Phi) \simeq \SL_2(\F_3)/\{\pm 1\} \simeq A_4$, proving the first statement.
Now, the condition 3 of \cite[Proposition~2.4]{Diamond} implies the second statement.
\end{proof}

\begin{lemma} \label{L:lemma55} 
The degree $d=[K:\Q_2]$ satisfies $d\in \lbrace 24r, 48r\rbrace$.
\end{lemma}

\begin{proof} We have $[F:\Q_2]=12$ or $24$ by Lemma~\ref{L:lemma54}.
Since $\Q_2(\mu_p) \subset K$ we have $24r \mid d$.
Now, a similar argument to Lemma~\ref{L:lemma51} shows that $[K:F]\leq 2r$. 
Then $d\leq 48r$ and the result follows.
\end{proof}

\begin{lemma} \label{L:lemma56} 
We have $[L:\Q_2]\in\lbrace 24, 48\rbrace$.
\end{lemma}

\begin{proof} As in Lemma~\ref{L:lemma53}, we have $[M(E_p):M]=\delta$ and by lemmas~\ref{L:lemma35}~and~\ref{L:lemma55}
we conclude $[L:\Q_2]\geq 12$. We have $[L:\Q_2]\neq 16$, otherwise $d$ would be $8r$, $16r$ or  $32r$ by Lemma~\ref{L:lemma35}, which is false by Lemma~\ref{L:lemma52}. Finally, $[L:\Q_2]\neq 12$, because $\GL_2(\F_3)$ has not normal subgroups of order $4$. The result follows.
\end{proof}

\begin{corollary} \label{C:cor4}
One has $\mu_3\subseteq K$.
\end{corollary}

\begin{proof} The result is clear when $[L:\Q_2]=48$ i.e. if $L=M$. 

Suppose $[L:\Q_2]=24$. The only normal subgroup of order $2$ of $\GL_2(\F_3)$ is its center $C$. It is contained in $\SL_2(\F_3)$ and $\GL_2(\F_3)$ contains a unique subgroup isomorphic to $\SL_2(\F_3)$. We deduce that $C$ is contained in the inertia subgroup of $\Gal(M/\Q_2)$, so the extension $M/L$ is ramified. Therefore, 
$L/\Q_2$ is not totally ramified, which implies $\mu_3 \subset L \subset K$.
\end{proof}

We can now complete the proof in the case $e=24$.

1) Suppose $r$ odd. Then $\mu_3 \not\subset \Q_2(\mu_p)$. Since $\Q_2(\mu_3)\subseteq K$ by Corollary~\ref{C:cor4}, we deduce that the degree of the maximal unramified subfield of $K$ is at least~$2r$. 
This implies $d\geq 48r$, so $d=48r$ by Lemma~\ref{L:lemma55}.

2) Suppose $r$ even. Then $\mu_3 \subset \Q_2(\mu_p)$. From Lemma~\ref{L:lemma54} we have that $F/\Q_2$ has 
ramification index 12 and degree $[F:\Q_2]=12$ or $[F:\Q_2]=24$. Write $d'=[K:F]$. 
 
Suppose $[F:\Q_2]=24$. Let us show that $d'=r$. 
The Galois group of $F/\Q_2$ is isomorphic to $S_4$ by Lemma~\ref{L:lemma54}. 
This extension is not totally ramified, 
thus $\Q_2(\mu_p)\cap F=\Q_2(\mu_3)$. 
The Galois groups $\Gal(F(\mu_p)/F)$ and $\Gal(\Q_2(\mu_p)/\Q_2(\mu_p)\cap F)$ being isomorphic, we conclude that
$$[F(\mu_p):F]=\frac{r}{2}.$$
This  leads to the equality
$$\frac{d'}{\gcd(d',2)}=\frac{r}{2}.$$
The ramification index of $K/F$ is equal to $2$, so $2$ divides $d'$. We obtain $d'=r$, thus $d=24d'=24r$.

Suppose $[F:\Q_2]=12$.  
We have $$\frac{d'}{\gcd(d',2)}=[F(\mu_p):F]\leq r.$$
So $d'\leq 2r$, which implies $d\leq 24r$. Then, by Lemma~55, $d=24r$, completing the proof.

{}

\end{document}